\documentclass{article}   

\usepackage{Preamble}

\usepackage{amsmath} 
\usepackage{amssymb} 
\usepackage{mathtools}
\usepackage{color}
\usepackage{enumerate}
\usepackage{graphicx} 
\usepackage{bm}

\author{Fatimah Al Saleh\thanks{		 
Email: {\tt\small fatimah.saleh@kaust.edu.sa, fhalsaleh@kfu.edu.sa}\ ,
	 }\ 
, Tigran Bakaryan\thanks{		 
Email: {\tt\small tigran.bakaryan@kaust.edu.sa }\ ,
	 }\ 
, Diogo A. Gomes\thanks{		 
		Email: {\tt\small diogo.gomes@kaust.edu.sa}\ ,
	 }\ 
	  , and  Ricardo Ribeiro\thanks{
	  	Email: {\tt\small ricardo.ribeiro@kaust.edu.sa} 
	}
}

\title{First-order mean-field games on networks and Wardrop equilibrium} 

\begin{document}

\maketitle

\begin{abstract}
%

Here, we examine the Wardrop equilibrium model on networks with flow-dependent costs and its connection with stationary mean-field games (MFG).
In the first part of this paper, we present the Wardrop and the first-order MFG models on networks. 
Then, we show how to reformulate the MFG problem into a Wardrop problem and prove that the MFG solution is the Wardrop equilibrium for the corresponding Wardrop problem. 
Moreover, we prove that the solution of the MFG problem can be recovered using the solution to the associated Wardrop problem. 
Finally, we study the cost properties and the calibration of MFG with Wardrop travel cost problems. 
We describe a novel approach to the calibration of MFGs.  Further, we show that even simple travel costs can give rise to non-monotone MFGs.

\end{abstract}


\section{Introduction}

Models for flows on networks arise in the study of traffic and pedestrian crowds.
These models encode congestion effects, the behavior and preferences of agents, such as aversion to crowds and their attempts to minimize travel time.  
A well-studied  stationary model is the Wardrop model introduced in \cite{Wardrop52}.
In this model, the cost of crossing a network edge  depends on the agents' flow in that edge.
Agents minimize their costs by taking into account that flow.  Multiple authors have studied Wardrop equilibria;  
	see, for example, the survey \cite{CJSMN}.
	In the context of optimal transport, Wardrop equilibrium was addressed in \cite{CarlierSant2011}
	and \cite{CarlierJimSant2008}. 
	The Wardrop model was initially formulated on a directed network.
	However, in problems such as pedestrian networks, the network is a priori undirected. Accordingly, the Wardrop model is
	not directly applicable. 
Further, the Wardrop model looks at a single edge as an aggregate entity. Thus, it does not describe the microstructure in the edge. 	
	Hence,    in  \cite{GomesMS19} authors introduced a mean-field game (MFG) model on undirected networks that attempts to model and address these matters. 
	

 MFG theory was introduced in \cite{ll1, Caines1}
to describe the  dynamics of systems with a large number of rational agents. 
In these models, agents seek to optimize an individual functional that depends on both their actions and the distribution of the other agents.
	A MFG model is determined by a backward-in-time Hamilton-Jacobi equation (HJ) coupled with  a forward-in-time Fokker-Planck (FP) or transport equation. The HJ equation  describes the optimal behavior of an agent; the FP equation governs the distribution of the agents.
 MFG models  have been used to address pedestrian flows \cite{MR3199781},  crowds \cite{Lachapelle10}, population dynamics
\cite{ MR2861584}, building evacuation problems \cite{MR3708886}, and in the study of traffic flows \cite{Bauso20143469}. 

	Recently, there has 
	been a broad interest in studying nonlinear PDEs on networks due to their applications in traffic and pedestrian dynamics.  Various notions of viscosity  solutions to HJ equations on networks were studied in \cite{AchdouCamilliCutriTchou, SchiebornCamilli,HJ_Network_2013_Zidani}. Later in \cite{CM234}, the authors proved the equivalency of the   viscosity  solutions of HJ equations  on networks.  The existence and uniqueness of viscosity solutions to   Eikonal equations on
	networks are addressed in \cite{CamilliMarchiSchieborn}.
More recently, progress on junction problems has been obtained in \cite{LionsSouganidisJunctionsI, LionsSouganidisJunctionsII}. 
Stationary HJ equations on networks were considered
in \cite{SiconolfiSorrentino} and \cite{RenatoHectorNetworks}.
Concerning transport phenomena on networks, the survey paper \cite{bellomo2011modeling} examines a substantial number of results. 
More recent works in this direction include  \cite{MR3553143, CamilliMaioTosin}.


Several authors in the MFG community examined MFGs on networks, in particular for second-order problems. For example,  stationary second-order MFGs were studied in  
\cite{cacacecamillimarchi, camillimarchi,MR3985393}.
The time-dependent case was studied in  
\cite{camillinet, AchdouDaoLeyTchou}. 
However, prior methods for MFGs on networks are not valid for first-order MFGs, where
a distinct set of phenomena occurs that includes the loss of smoothness for Hamilton-Jacobi (HJ) equations and lack of continuity for the value function at the vertices.
First-order MFGs on networks were considered in \cite{Bagagiolo2022B}, 
\cite{Bagagiolo2022}, \cite{MR4322154} and \cite{MR2886014}, in particular in the context of 
optimal visiting problems where agents have multiple targets. The methods in those papers, 
because are applied to general time-dependent problems are quite different from ours, where we take full advantage of the stationary nature of the game. Recently, in
\cite{AchdouManuciMarchiTchou2022} time-dependent MFGs on networks were examined in the light of weak solutions. 
First-order deterministic MFGs with control on the acceleration were studied in \cite{MR4102464}.
As far as the authors are aware,
there is no systematic approach for solving stationary first-order MFGs on networks, despite their relevance in applications that cannot be modeled by second-order MFGs, such as vehicular networks or dense crowds.

The relation between Wardrop equilibrium and MFGs on networks was first observed in \cite{GomesMS19}. 
However, the precise correspondence between these two models was not established there. 
 Here, we address this problem; that is, how to convert a MFG on a network into a standard Wardrop problem and further use the solution of the corresponding Wardrop problem to solve the original MFG.

As we mentioned before, the standard Wardrop equilibrium is formulated on a directed network.
In this paper,
we consider a Wardrop model,  where agents enter through a finite set of vertices with a prescribed current intensity and leave through a set of exit vertices where they pay an exit cost.  Although  our MFG model is defined on undirected networks, 
the preceding problem is closely related to a MFG model on networks. In our MFG model, the entry currents and the exit costs are given, but  the current direction in each edge is not prescribed a priori.
We discuss an explicit way to transform one problem into the other.

In Section \ref{Wmodel}, 
 we begin by outlining the main definitions and results about Wardrop equilibria. We define the stationary Wardrop model on a directed network, with flow-dependent travel costs on the edges. 
 Furthermore, we prove the uniqueness of the Wardrop equilibrium (see Theorem \ref{uni}).
Next, we  describe our MFG model on networks. We start by considering a MFG model on a single edge and recalling the current method introduced in \cite{Gomes2016b, GNPr216} (see Section \ref{costsection}). First, we define a current dependent cost (see Definition \ref{def-cost-in-MFG}) on the edge for the MFG model. 
 Then, we establish the connection between the cost and the solution to the  MFG system (see, for instance, \eqref{MFG}). 
In Section \ref{MFG model}, we continue our discussion about MFG models on networks.
In each edge $e_k$, our model is given by a 
one-dimensional, stationary, first-order MFG system
\begin{align}\label{MFG}
  \begin{cases}
    H_k(x,u_x(x),m(x))=0 \\ 
    (-m(x)D_pH_k(x,u_x(x),m(x))_x=0
  \end{cases}
\end{align}
together with boundary conditions at the vertices.
Here, $u$ is the value function for an agent and $m$ is the probability density of agents.
The first equation in \eqref{MFG} is the Hamilton-Jacobi (HJ) equation
and the second equation is the Fokker-Planck (FP) or transport equation. 

The first main contribution of this paper is the reformulation of
 the MFG problem into a Wardrop problem and the development of a method to recover the MFG solution from the solution to the  corresponding Wardrop model. 
In Section \ref{MFGtoW}, we start with the undirected network of the MFG model and build a new directed network on which we define a corresponding Wardrop problem.
Then, we prove that the solution    
of the MFG model is the Wardrop equilibrium
(see Theorem \ref{soltheorem}). Moreover, Proposition \ref{unicurr} proves that the associated Wardrop problem has a unique solution. This result gives the uniqueness of the solution  
to the original MFG problem (see Proposition \ref{pro--unique}). To complete the correspondence between MFG and Wardrop models, in Section \ref{WtoMFg}, we show how to recover the solution to the MFG problem from the associated Wardrop equilibrium (see Theorem \ref{theorem-Wardrop-to-MFG}). 

In MFG models, the cost in an edge can change from point to point. In contrast, in the Wardrop model, the cost is a function of the edges; each edge is treated as a homogeneous entity. Thus, the conversion between MFG and Wardrop models averages the microscopic effects in the MFG into a macroscopic cost that determines the Wardrop.   
Section \ref{costprop} comprises the second main result of the paper. There, we study how the microscopic properties of the MFG encoded in the Hamiltonian translate into macroscopic properties of the corresponding Wardrop model.  Moreover, we explore a class of Hamiltonians for which the associated cost function satisfies the assumptions required for the unique correspondence between MFG and Wardrop models (see the assumptions of Proposition \ref{unicurr}).


%

In a Wardrop equilibrium, the cost incurred by an agent on an edge is the travel time.
This time depends on the flow of agents and, thus, can be measured experimentally.
Therefore, in principle, 
these models are simple to calibrate: it is enough to have data on the flow rates and corresponding velocities. This calibration problem is more delicate for MFG models. 
Therefore, in Section \ref{sec:MFG_calibration}, we show how to calibrate a MFG model on a single edge using the notions and results of the previous sections. The problem is the following.
\begin{problem}\label{prob:p001}
Consider the cost $c$ for an agent to cross an edge. Suppose $c$ is  given as a function of the current and the direction of travel.   Find a MFG model whose cost coincides with $c$.
\end{problem}
Note that  the cost, $c$,
may not be the travel time, and as far as the authors know, before this work, there was no systematic approach to solving Problem \ref{prob:p001}.	
We rigorously state the preceding problem 
in Section \ref{sec:MFG_calibration} and solve it. 
In particular, we show how non-monotone MFG models arise as solutions to Problem \ref{prob:p001} for natural costs.	
Hence, non-monotone MFG models may be needed to model general traffic problems. 
In non-monotone MFG, agents act as if they want to be in congested areas, which is surprising from a modeling point of view.   
 These non-monotone problems are not well studied, and further research should be pursued.
The present work suggests several future directions. For example, the study of non-monotone MFGs and their general properties and the extension of our models to the dynamic setting. The lack of monotonicity for these MFG models makes it possible that new phenomena arise, such as multiple equilibria or instability. We believe these to be relevant questions for further research.

\section{Wardrop equilibrium model on networks}\label{Wmodel}

We consider a steady-state model of agents traveling through a network with flow-dependent travel costs on the edges.
Following the presentation of the model's elements, we discuss an equilibrium concept due to Wardrop \cite{Wardrop52}. Later, in Section \ref{MFGtoW}, we show how this concept is related to the MFG model presented in Section \ref{MFG model}.
This model consists of the following:
\begin{enumerate}

\item A finite \textbf{directed} network (graph), $\tilde{\Gamma}=(\tilde{E},\tilde{V})$, where
$\tilde{E} = \{ \tilde{e}_k : k \in \{1,2,\dots,\tilde{n}\} \}$ is the set of edges
and $\tilde{V}=\{\tilde{v}_i : i \in \{1,2,\dots, \tilde{m}\} \}$ is the set of vertices.
To any edge $\tilde{e}_k$, we associate the pair $(\tilde{v}_r,\tilde{v}_i)$ of its endpoints, which may be identified with $\tilde{e}_k$ if there is no ambiguity.

\item The \emph{current} or the flow is the number of agents crossing a given point per unit of time.
The current in the edge $\tilde{e}_k$ is denoted by $\tilde{\jmath}_k$.

The current on the network is the $\tilde{n}$-dimensional vector $\bm{\tilde{\jmath}}=(\tilde{\jmath}_1,\dots,\tilde{\jmath}_{{\tilde{n}}})$.

\item The \emph{travel cost} in each edge $\tilde{e}_k$ is given by a function $ {\tilde{c}}_k : \tilde{E} \rightarrow \mathbb{R}$.
The cost is local if $ {\tilde{c}}_k$ depends only on $\tilde{\jmath}_k$.

We denote the vector of the costs in the edges by $\bm{\tilde{c}}(\bm{\tilde{\jmath}})=(\tilde{c}_1(\bm{\tilde{\jmath}}),\dots,\tilde{c}_{\tilde{n}}(\bm{\tilde{\jmath}}))$. Observe that 
$\langle\bm{\tilde{c}}(\bm{\jmath}),\bm{\jmath}\rangle:= \sum_{k=1}^{\tilde{n}} \tilde{c}_k(\bm{\jmath})j_k $ is the social cost per unit of time corresponding to the distribution of currents $\bm{\jmath}$.

\item  Agents enter the network through $\tilde{\lambda}$ {\em entrance vertices} and exit it through $\tilde{\mu}$
{\em exit vertices} (disjoint from the entrance vertices). 
For convenience, we assume that the last $\tilde{\mu}$ vertices in $\tilde{V}$ are the exit vertices. 
Furthermore, we suppose that entrance and exit vertices have incidence $1$. As discussed in the remark below, this assumption entails no loss of generality. 

\begin{remark}\label{inoutedges}
Suppose a vertex with an incidence of more than $1$ is labeled as an entrance. In that case,
we attach an auxiliary entrance edge to it, relabeling its extra vertex as the new entrance vertex.	
Similarly, relabeling is done for an exit vertex by adding an auxiliary {\em exit edge}.  
\end{remark}

\item A flow of agents, the entry current $\tilde{\bm{\iota}}=(\tilde{{\iota}}_1,\dots,\tilde{{\iota}}_{\tilde{\lambda}}) >0$, is prescribed at the entrance vertices.
The entry currents in the other vertices are zero.
This information is encoded in a $(\tilde{m}-\tilde{\mu})$-dimensional vector $\tilde{B}$.
Each component of $\tilde{B}$ corresponds to a non-exit vertex. 

\item At the $\tilde{\mu}$ exit vertices, agents pay an exit cost $\tilde{\bm{\phi}}$. Here, we assume that this exit cost vanishes. As the following remark explains, this assumption entails no loss of generality. 

\begin{remark}\label{exitremark}
If the exit cost at an exit vertex is nonzero, we attach an auxiliary exit edge to it, relabeling its extra vertex as the new exit vertex.
In the auxiliary exit edge, the travel cost is the exit cost, and at the new exit vertex, the exit cost is zero.

\end{remark}

\end{enumerate}

Let $\tilde{K}$ be the $(\tilde{m}-\tilde{\mu}) \times \tilde{n}$ Kirchhoff matrix, obtained by removing $\tilde{\mu}$ lines corresponding to the exit vertices of $\tilde{\Gamma}$ from the \emph{incidence matrix} of $\tilde{\Gamma}$, i.e.,
\begin{equation}\label{K}
\tilde{K}_{ik}=
\begin{cases}
1 \ \ \text{if} \ \tilde{e}_k=(\tilde{v}_i,\tilde{v}_r) \\
-1\ \ \text{if} \ \tilde{e}_k=(\tilde{v}_r,\tilde{v}_i)\\
0 \ \ \text{if} \ \tilde{v}_i \not\in \tilde{e}_k,
\end{cases}
\end{equation}
where $ i \in \{1,2,\dots, \tilde{m}-\tilde{\mu}\}$ and $k \in \{1,2,\dots,\tilde{n}\} $.
The rows in $\tilde{K}$ correspond to non-exit vertices of the network, and the columns correspond to the edges. 

\begin{definition}
A distribution of currents, ${\bm{\tilde{\jmath}}} \geq 0$, is \emph{admissible} if 
\begin{equation}\label{admis}
\tilde{K}\bm{\tilde{\jmath}}=\tilde{B}.
\end{equation}
The set of all admissible distributions of currents is denoted by $\mathcal{A}$.
\end{definition}
\begin{remark}
The $\tilde{m}-\tilde{\mu}$ equations in \eqref{admis} correspond to  Kirchhoff's law for the non-exit vertices.
\end{remark}

Next, following Smith (\cite{Smith1979traffic}), we define the Wardrop equilibrium.
\begin{definition}\label{defWard}
A distribution of currents ${\bm{\tilde{\jmath}}}^* \in \mathcal{A}$ is a \emph{ Wardrop equilibrium} if, for all $\bm{\tilde{\jmath}} \in \mathcal{A}$,
\begin{equation}\label{Wardrop}
\langle \bm{\tilde{c}}({\bm{\tilde{\jmath}}}^*),{\bm{\tilde{\jmath}}}^*-\bm{\tilde{\jmath}} \rangle \leq 0.
\end{equation}
\end{definition}
In Section \ref{WtoMFg}, we prove that Definition \ref{defWard} means that in a Wardrop equilibrium, any current-carrying path is optimal because it minimizes the travel cost to an exit. 

We use the notion of monotonicity, defined below, to state and prove a result on the uniqueness of Wardrop equilibria.
\begin{definition}
A cost $\bm{\tilde{c}}$ is \emph{monotone} if, for any $\bm{\tilde{\jmath}_1}, \bm{\tilde{\jmath}_2} \in \mathcal{A}$,
\begin{equation}\label{mono}
\langle \bm{\tilde{c}}(\bm{\tilde{\jmath}_1})-\bm{\tilde{c}}(\bm{\tilde{\jmath}_2}),\bm{\tilde{\jmath}_1}-\bm{\tilde{\jmath}_2} \rangle \geq 0.
\end{equation}
If $\bm{\tilde{\jmath}_1} \neq \bm{\tilde{\jmath}_2}$, the inequality in \eqref{mono} is strict; in that case, we say that $\bm{\tilde{c}}$ is \emph{strictly monotone}.
\end{definition}
\begin{example}
If $\tilde{c}_k$ depends only on $\tilde{\jmath}_k$ and is an increasing function,
then $\bm{\tilde{c}}$ is monotone.
\end{example}
Strict monotonicity of the cost is sufficient for uniqueness.
\begin{theorem}[Uniqueness of Wardrop equilibrium]\label{uni}
Suppose the cost $\bm{\tilde{c}}$ is strictly monotone on  $\mathcal{A}$.
Then, there is at most one Wardrop equilibrium.
\end{theorem}
\begin{proof}
Suppose ${\bm{\tilde{\jmath}_1}}$ and ${\bm{\tilde{\jmath}_2}}$ are Wardrop equilibria.
Then, for any $\bm{\tilde{\jmath}} \in \mathcal{A}$, we have
\begin{equation}
\langle \bm{\tilde{c}}({\bm{\tilde{\jmath}_1}}),{\bm{\tilde{\jmath}_1}}-\bm{\tilde{\jmath}} \rangle \leq 0
\ \text{and} \
\langle \bm{\tilde{c}}({\bm{\tilde{\jmath}_2}}),{\bm{\tilde{\jmath}_2}}-\bm{\tilde{\jmath}} \rangle \leq 0.
\end{equation}
Accordingly,
$$
\langle \bm{\tilde{c}}({\bm{\tilde{\jmath}_1}})-\bm{\tilde{c}}({\bm{\tilde{\jmath}_2}}),{\bm{\tilde{\jmath}_1}}-{\bm{\tilde{\jmath}_2}} \rangle \leq 0,
$$
and because $\bm{\tilde{c}}$ is strictly monotone, ${\bm{\tilde{\jmath}_1}}={\bm{\tilde{\jmath}_2}}$.
\end{proof}

Next, we discuss a necessary condition for the existence of a Wardrop equilibrium.

\begin{definition}
A current distribution $\bm{\tilde{\jmath}_0} \geq 0$ on $\tilde{\Gamma}$ is a \emph{current loop} if it is a nontrivial solution of $\tilde{K}\bm{\tilde{\jmath}_0}=0$.
\end{definition}

\begin{proposition}\label{loopcost}
If $\bm{\tilde{\jmath}^*}$ is a Wardrop equilibrium, then
$$
\langle \bm{\tilde{c}}(\bm{\tilde{\jmath}^*}),\bm{\tilde{\jmath}_0} \rangle \geq 0,
$$
for any current loop, $\bm{\tilde{\jmath}_0}$.
\end{proposition}
\begin{proof}
Let $\bm{\tilde{\jmath}^*}$ be a Wardrop equilibrium and let $\bm{\tilde{\jmath}_0}$ be a current loop. Set 
$$
\bm{\tilde{\jmath}}(\varepsilon)=\bm{\tilde{\jmath}^*}+\varepsilon \bm{\tilde{\jmath}_0}.
$$
Since $\bm{\tilde{\jmath}_0}$ is a loop, $\bm{\tilde{\jmath}}(\varepsilon) \in \mathcal{A}$, for $\varepsilon \in \mathbb{R}^+_0$.
Condition \eqref{Wardrop}, for the particular $\bm{\tilde{\jmath}}(\varepsilon)$, implies that
$$
\langle \bm{\tilde{c}}(\bm{\tilde{\jmath}^*}),\bm{\tilde{\jmath}^*}-(\bm{\tilde{\jmath}^*}+\varepsilon \bm{\tilde{\jmath}_0}) \rangle =
-\varepsilon \langle \bm{\tilde{c}}(\bm{\tilde{\jmath}^*}),\bm{\tilde{\jmath}_0} \rangle \leq 0, \ \forall \varepsilon >0.
$$
Thus,
$
\langle \bm{\tilde{c}}(\bm{\tilde{\jmath}^*}),\bm{\tilde{\jmath}_0} \rangle \geq 0.
$
\end{proof}

Next, we examine the structure of the Wardrop equilibrium in a specific case relevant to MFGs.
The networks arising in MFGs are undirected. 
We associate to these networks a network for which undirected edges correspond to pairs of directed edges with opposite orientations.  
Thus, these networks have as a primary building block the loop network in Figure \ref{loop}.
\begin{IMG}
{"loop",Graph[{Labeled[1 \[DirectedEdge] 2, Subscript[OverTilde["e"], k]],
Labeled[2 \[DirectedEdge] 1, Subscript[OverTilde["e"], l]]},
VertexLabels -> {1 -> Subscript[OverTilde["v"], r],
2 -> Subscript[OverTilde["v"], i]}, GraphLayout -> "RadialDrawing"](*\]\]*)}	
\end{IMG}
\begin{figure}[ht]
\centering
\includegraphics[scale=1]{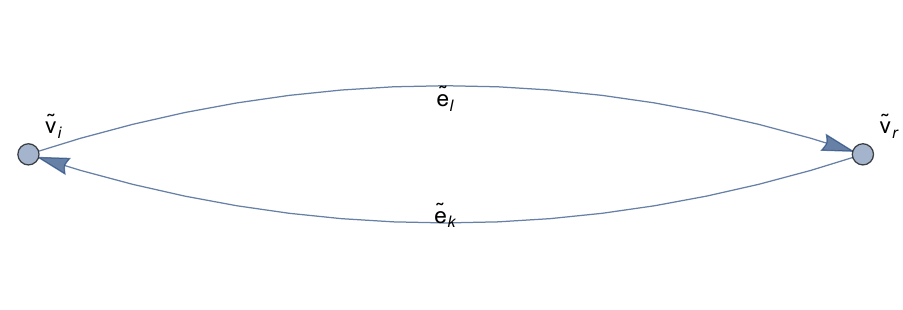}
\caption{Loop subnetwork.}
\label{loop}
\end{figure}
\begin{proposition}\label{pro-no-loops-in-Wardrop}
Consider a network containing the edge
$\tilde{e}_k=(\tilde{v}_r,\tilde{v}_i)$ and the reverse edge $\tilde{e}_l=(\tilde{v}_i,\tilde{v}_r)$.
Let $\tilde{\jmath}_k \geq 0$ be the current in $\tilde{e}_k$, and $\tilde{\jmath}_l \geq 0$ be the current in $\tilde{e}_l$.
Let $\tilde{c}_k(\bm{\tilde{\jmath}})$ be the cost in $\tilde{e}_k$, and $\tilde{c}_l(\bm{\tilde{\jmath}})$ be the cost in $\tilde{e}_l$,
where $\bm{\tilde{\jmath}}=(\tilde{\jmath}_k,\tilde{\jmath}_l)$.
Furthermore, suppose that the costs $\tilde{c}_k$ and $\tilde{c}_l$ satisfy the following condition
\begin{equation}\label{poscost}
\tilde{c}_k(\bm{\tilde{\jmath}})+\tilde{c}_l(\bm{\tilde{\jmath}}) >0.
\end{equation}
Then, any Wardrop equilibrium, ${\bm{\tilde{\jmath}}}^*=(\tilde{\jmath}^*_k,\tilde{\jmath}^*_l)$, satisfies the complementary condition
\begin{equation}\label{WardComp}
\tilde{\jmath}^*_k \cdot \tilde{\jmath}^*_l = 0, \ \ k\neq l.
\end{equation}
\end{proposition}
\begin{proof}
We prove the theorem for a network that has Figure \ref{loop} as a subnetwork.
To simplify the discussion below, we will omit all the edges and currents of this network that are not depicted in Figure \ref{loop}. 
Assume that ${\bm{\tilde{\jmath}}}^*=(\tilde{\jmath}^*_k,\tilde{\jmath}^*_l)$ is a Wardrop equilibrium and does not satisfy \eqref{WardComp}, i.e.,
without loss of generality $\tilde{\jmath}^*_k > \tilde{\jmath}^*_l > 0$.
Consider the current loop $\bm{\tilde{\jmath}_0}=(\tilde{\jmath}^*_k-\tilde{\jmath}^*_l,0)$, then by \eqref{Wardrop}, we have
\begin{equation*}
\langle \bm{\tilde{c}}({\bm{\tilde{\jmath}}}^*),{\bm{\tilde{\jmath}}}^*-\bm{\tilde{\jmath}_0} \rangle=
\langle (\tilde{c}_k({\bm{\tilde{\jmath}}}^*),\tilde{c}_l({\bm{\tilde{\jmath}}}^*)), (\tilde{\jmath}^*_l,\tilde{\jmath}^*_l) \rangle \leq 0
\end{equation*}
which contradicts the assumption \eqref{poscost}.
Hence, ${\bm{\tilde{\jmath}}}^*$ cannot be a Wardrop equilibrium.
\end{proof}


Next, we analyze the monotonicity of the cost on a network containing the subnetwork in Figure \ref{loop}.
\begin{example}\label{exmono}
The monotonicity condition for the subnetwork in Figure \ref{loop} is
$$
\langle (\tilde{c}_1(\bm{\tilde{\jmath}}),\tilde{c}_2(\bm{\tilde{\jmath}}))-(\tilde{c}_1(\bm{\hat{\jmath}}),\tilde{c}_2(\bm{\hat{\jmath}})), \bm{\tilde{\jmath}}-\bm{\hat{\jmath}} \rangle \geq 0,
$$
where $\bm{\tilde{\jmath}}=(\tilde{\jmath}_1,\tilde{\jmath}_2)$ and $\bm{\hat{\jmath}}=({\hat{\jmath}}_1,{\hat{\jmath}}_2)$.
Then, we get
$$
(\tilde{c}_1(\bm{\tilde{\jmath}}) -\tilde{c}_1(\bm{\hat{\jmath}}), \tilde{\jmath}_1-\hat{\jmath}_1 )
+ ( \tilde{c}_2(\bm{\tilde{\jmath}}) -\tilde{c}_2(\bm{\hat{\jmath}}), \tilde{\jmath}_2-\hat{\jmath}_2 )
\geq 0.
$$
Suppose that
$\tilde{c}_1(\bm{\tilde{\jmath}})= \hat{c}_1(\tilde{\jmath}_1 - \tilde{\jmath}_2)$ and $\tilde{c}_2(\bm{\tilde{\jmath}})= \hat{c}_2(\tilde{\jmath}_2 - \tilde{\jmath}_1)$.
Then, we have
\begin{equation*}
(\hat{c}_1(\tilde{\jmath}_1 - \tilde{\jmath}_2)-
\hat{c}_1(\hat{\jmath}_1 - \hat{\jmath}_2),
\tilde{\jmath}_1-\hat{\jmath}_1) +
(\hat{c}_2(\tilde{\jmath}_2 - \tilde{\jmath}_1)-
\hat{c}_2(\hat{\jmath}_2 - \hat{\jmath}_1),
\tilde{\jmath}_2-\hat{\jmath}_2) \geq 0.
\end{equation*}
\end{example}

The previous example can be generalized to networks built out of subnetworks like the one in Figure \ref{loop}.
The sum in the monotonicity condition
$$
\sum_k (\tilde{c}_k({\bm{\tilde{\jmath}}})-\tilde{c}_k({\bm{\hat{\jmath}}}))(
{\tilde{\jmath}}_k-{\hat{\jmath}}_k) \geq 0
$$
can be organized in related pairs of edges.


\section{Mean-field games problem in a single edge}\label{costsection}
We start the discussion of MFGs on networks  by analyzing a single edge.

\subsection{MFG system in an edge}\label{singlemfg}
Consider an edge that is identified with the interval $[0,1]$. In this edge, the Hamiltonian 
is a function
$H:[0,1]\times \mathbb{R}\times\mathbb{R}_0^+\to \mathbb{R}$, 
smooth and convex in the second variable. 
The two unknowns that determine the MFG are the player's density $m:[0,1]\to \mathbb{R}_0^+$ and the value function $u:[0,1]\to\mathbb{R}$. These two functions solve the MFG
\begin{equation}\label{MFGgeneral}
\begin{cases}
H\left(x,u_x(x),m(x)\right)=0 \\
\left(-m(x)D_pH(x,u_x(x),m(x)\right)_x=0.
\end{cases}
\end{equation}
The first equation is the Hamilton-Jacobi (HJ) equation;  the second equation 
is the Fokker-Plank or transport equation.

From the second equation in \eqref{MFGgeneral}, the current 
\begin{equation}\label{ABC}
    j=-mD_pH\left(x,u_x,m\right)
\end{equation} is constant along the edge.
Thus, in principle, we can solve \eqref{ABC} for $u_x$ and replace it in the first equation.
Accordingly, we obtain an algebraic equation for $m$; see Proposition \ref{propform} for details. 
This equation defines a distribution of players $m(x,j)$. 
This procedure is known as the current method introduced in \cite{GNPr216}.

Next, we illustrate this method
for a family of Hamiltonians that is a prototype for MFG with congestion. 
\begin{example}\label{ExampleMFG}
Let
$$
H(x,p,m)=\frac{|p|^2}{2m^{\alpha}} + V(x) - g(m),
$$
where $g$ is an increasing function, $V$ is a smooth function and $0\leq \alpha \leq 2$ is the congestion strength.
Then, \eqref{MFGgeneral} becomes
\begin{equation}\label{mfgs}
\begin{cases}
\frac{|u_x|^2}{2m^{\alpha}} + V(x) = g(m)\\
(-m^{1-\alpha}u_x)_x=0.
\end{cases}
\end{equation}
There are two special values of the parameter $\alpha$,
$\alpha=0$ for the uncongested MFG, and $\alpha=1$, which corresponds to the critical congestion model.
In this last case, the system decouples, $u$ is a linear function, and $m$ can be obtained by solving the first equation in \eqref{mfgs}.

Because the current, $j=-m^{1-\alpha}u_x$, is constant, we transform \eqref{mfgs} into the following algebraic system 
\begin{equation}\label{current mfg}
\begin{cases}
\frac{j^2}{2m^{2-{\alpha}}}+V(x)=g(m)\\
j\int_0^1m^{\alpha-1}dx=u(0)-u(1),\\
\end{cases}
\end{equation}
provided $m>0$.
We refer to the last equation in \eqref{current mfg} as the \emph{edge equation}.
In this equation, either $u(0)$ or $u(1)$ is given by boundary conditions.
The sign of the current indicates the direction in which the agents travel, from $0$ to $1$ when $j>0$ and from $1$ to $0$ when $j<0$.
When $j>0$, the edge equation determines $u(0)$ as a function of $u(1)$, whereas when $j<0$, this equation determines $u(1)$ as a function of $u(0)$.
This depends on how boundary conditions are specified, as we discuss next.
\end{example}

\subsection{Travel cost in an edge}

Consider an edge and, as before, assume that this edge is identified with the interval $[0,1]$ and the vertices with $\{0,1\}$.
Given the current $j$ in this edge, there are two costs.
$c_{01}(j)$ is the cost of moving from left to right (i.e.,
from $0$ to $1$); $c_{10}(j)$ is the cost of moving from right to left (i.e.,
from $1$ to $0$).
To define these costs, we introduce the Lagrangian
\begin{equation*}
    L(x,v,m) =\sup_{p}{-p v -H(x,p,m)}.
\end{equation*}
The cost $c_{xy}$ between two points $0\leq x\leq 1$ and $0\leq y\leq 1$ is 
$$
c_{xy}= \min_{\substack{x,T \\ x(0)=x\\ x(T)=y}}
\int_0^T L(x,v,m) dt.
$$
For $0< x< 1$ and $0< y<1$,
the HJ equation satisfies the dynamic programming principle 
$$
u(x)= \inf_y c_{xy}+u(y). 
$$
When $x=0,1$ and $y=1-x$, the cost and the boundary conditions are related as follows. 
 
\begin{definition}\label{def-cost-in-MFG}
Given a distribution of players $m(x,j)$. For $j \in \mathbb{R}$, 
the optimal cost of traveling from $0$ to $1$ on the edge is 
\begin{equation}\label{costdef}
c_{01}(j)= \min_{\substack{x,T \\ x(0)=0\\ x(T)=1}}
\int_0^T L(x,\dot{x},m(x,j)) dt,
\end{equation}
where the minimum is taken over all Lipschitz trajectories and all possible terminal times.
The cost $c_{10}$ is defined analogously:
\begin{equation}\label{costdef2}
    c_{10}(j)= \min_{\substack{x,T \\ x(0)=1\\ x(T)=0}}
    \int_0^T L(x,\dot{x},m(x,j)) dt.
    \end{equation}
\end{definition}

The agent at a vertex of an edge has two options: leaving this edge or crossing it to get to the opposite vertex of the same edge. 
This is encoded in the following inequalities
\begin{equation}\label{optforc}
\begin{cases}
u(1)\leq c_{10}(j)+u(0), \\
u(0)\leq c_{01}(j)+u(1).
\end{cases}
\end{equation}
Moreover, the self-consistency in MFGs requires that the direction in which the agent moves to be aligned with the direction of the current.
This is reflected in the following condition
\begin{equation}\label{BC}
\begin{cases}
j>0 \Rightarrow u(0)=c_{01}(j)+u(1), \\
j<0 \Rightarrow u(1)=c_{10}(j)+u(0).
\end{cases}
\end{equation}

\begin{remark}
The optimality conditions in \eqref{optforc} imply the following local compatibility condition
$$
-c_{01}(j) \leq c_{10}(j).
$$
This can be interpreted as loops having a non-negative cost
$$
c_{01}(j)+c_{10}(j) \geq 0.
$$
\end{remark}
\begin{proposition}
If the costs in \eqref{optforc} are non-negative, then the inequalities in \eqref{optforc} are redundant.
\end{proposition}
\begin{proof}
The optimality conditions \eqref{optforc} imply,
$$
-c_{01}(j) \leq u(1)-u(0) \leq c_{10}(j).
$$
If $j \neq 0$, one of the two inequalities is an equality.
Since the costs are non-negative, then the other inequality holds.
\end{proof}

Now, we analyze the variational problems that define the costs $c_{01}$ and $c_{10}$.

Starting with $c_{01}$, we parametrize, in \eqref{costdef}, the velocity $v$ by the space coordinate $x$, so
$dx=v dt$
provided $\dot{x}>0$,
we get
\begin{equation}\label{Lv}
c_{01}(j)= \min \int_0^1 {\frac{L(x,v(x,j),m(x,j))}{v(x,j)}} dx,
\end{equation}
where the minimum is taken in the set, $\text{Lip}_x\left([0,1]\times \mathbb{R}\right)$, of functions $v:[0,1]\times \mathbb{R}\to \mathbb{R}^+$ that are Lipschitz in the first variable.
Similarly, we have 
\begin{equation}\label{Lvb}
c_{10}(j)= \min \int_0^1 -{\frac{L(x,v(x,j),m(x,j))}{v(x,j)}} dx,
\end{equation}
where the minimum is taken in the set $\text{Lip}_x\left([0,1]\times \mathbb{R}\right)$, of functions $v:[0,1]\times \mathbb{R}\to \mathbb{R}^-$ that are Lipschitz in the first variable.

\begin{proposition}\label{Propforv}

    If $L(x,v,m)$ is coercive in $v$ and $L(x,0,m)>0$, then the following holds
    
    \begin{enumerate}
    
    \item There exist two functions $v_+^*>0$ and $v_-^*<0$ that minimize \eqref{Lv} and \eqref{Lvb}, respectively. Moreover, these functions solve the Euler-Lagrange equation 
    \begin{equation}\label{L/v}
        -\frac{L(x,v(x,j),m(x,j))}{v(x,j)^2}+\frac{D_vL(x,v(x,j),m(x,j))}{v(x,j)}=0.
        \end{equation}
        
\item         
We have
    \begin{equation}\label{cost}
c_{01}(j)= \int_0^1 {D_vL(x,v_+^*(x,j),m(x,j))} dx , \text{for any } j\in \mathbb{R}.
\end{equation}
and 
\begin{equation}\label{cost10}
c_{10}(j)= - \int_0^1 {D_vL(x,v_-^*(x,j),m(x,j))} dx,\text{for any } j\in \mathbb{R}.
\end{equation}
        
 \item  If $v\mapsto L(x,v,m)$ is strictly convex in $v$, then $v_+^*$ is unique in $\mathbb{R}^+$, and $v_-^*$ is unique in $\mathbb{R}^-$.
\end{enumerate}          
  
    \end{proposition}
    \begin{proof}

    Because $L(x,0,m)>0$, we have
    $$ \lim_{v \rightarrow 0} \frac{L(x,v(x,j),m(x,j))}{v(x,j)}= +\infty,$$
    and by coercivity, we have
    $$ \lim_{v \rightarrow \infty} \frac{L(x,v(x,j),m(x,j))}{v(x,j)}= +\infty.$$
    Then, for every $x$, there exists a minimizer $v_+^*>0$ of $\frac{L(x,v,m)}{v}$
    which solves \eqref{L/v}.
  This pointwise minimizer is a minimizer of \eqref{Lv}.
    
By \eqref{L/v}, we have 
\begin{equation}\label{DvL}
D_vL(x,v_+^*(x,j),m(x,j))=\frac{L(x,v_+^*(x,j),m(x,j))}{v_+^*(x,j)}.
\end{equation}
Thus, we get \eqref{cost}.

    Now, we prove the uniqueness of the minimizer.
    Because \eqref{L/v} holds, 
    we have
    \begin{equation*}      
        \begin{aligned}
            D^2_{vv}\left(\frac{L}{v}\right) &= D_v\left(\frac{D_vL}{v}-\frac{L}{v^2}\right)\\
            &=\frac{D^2_{vv}L}{v} -2\frac{D_vL}{v^2}+2\frac{L}{v^3}\\
            &=\frac{D^2_{vv}L}{v}.
        \end{aligned}
    \end{equation*}
    The above identity with $D^2_{vv}L(x,v,m) >0 $ and $v_+^*>0$ implies that $v \mapsto\frac{L(x,v,m)}{v}$ is strictly convex.
    Therefore, the integrand of the variational problem is strictly convex, which gives the uniqueness of $v^*_+$.


Following the same steps, we can prove similar results for 
$v_-^*$.
\end{proof}

\begin{remark}
The cost $c_{01}$ in \eqref{cost} is defined for any $j \in \mathbb{R}$. In particular, when $j<0$, $c_{01}$ is the cost of moving against the current. 
\end{remark}



We now examine the relation between the cost and value functions.

\begin{proposition}\label{propc&u}

         For $j>0$, the cost  $c_{01}(j)$ is given by
  $$
  c_{01}(j)= \int_0^1  c_{01}(x,j) dx, 
  $$
   where
    $c_{01}(x,j)= -u_x$ and  solves the MFG system
        \begin{equation}\label{mfgsys}
        \begin{cases}
        H(x,-c_{01}(x,j),m(x,j))=0, \\
        -mD_pH(x,-c_{01}(x,j),m(x,j))=j.
        \end{cases}
        \end{equation}
 Similarly, for $j<0$, the cost  $c_{10}(j)$ is given by
  $$
  c_{10}(j)= \int_0^1 c_{10}(x,j) dx,
  $$
  where $ c_{10}(x,j)= u_x$ and  solves the MFG system
        \begin{equation*}
        \begin{cases}
        H(x,c_{10}(x,j),m(x,j))=0, \\
        -mD_pH(x,c_{10}(x,j),m(x,j))=j.
        \end{cases}
        \end{equation*}
\end{proposition}
\begin{proof}
Consider an edge with a current $j>0$.
By the fundamental theorem of calculus,
$$
u(0)=u(1)-\int_0^1{u_x} dx.
$$
By the first identity in \eqref{BC}, we have 
$$
u(0)=u(1)+c_{01}(j). 
$$
Hence, 
\begin{equation}\label{c01ux}
c_{01}(j)=-\int_0^1{u_x} dx.
\end{equation}
Thus, $-u_x$ can be regarded as a cost per unit length.
We define $c_{01}(x,j) = -u_x$. Therefore, $c_{01}(x,j)$ solves \eqref{mfgsys}, and we have
$$
c_{01}(j)=\int_0^1 c_{01}(x,j) dx.
$$
The proof for the case $j<0$ is similar. 
\end{proof}

\begin{proposition}\label{propform}
   If $H$ is strictly convex in $p$,  we have the following 
    \begin{enumerate}
   \item 
   \begin{equation}\label{uxform}
    u_x=-D_vL\left(x,\frac{j}{m},m\right);
     \end{equation}
   \item   $m(x,j)$ solves 
\begin{equation}\label{m}
H\left(x,-D_vL\left(x,\frac{j}{m},m\right),m\right)=0.
\end{equation}
  \end{enumerate}
\end{proposition}

\begin{proof}
Since $H$ is strictly convex in $p$, 
$ p \rightarrow D_pH(x,p,m)$
is invertible.
Moreover, if $D_vL(x,v,m)=-p$ and $v=-D_pH(x,p,m)$, we have
\begin{equation*}
p=-D_vL(x,-D_pH(x,p,m),m).
\end{equation*}
From \eqref{ABC}, we have
$$
D_pH(x,u_x,m)=-\frac{j}{m}.
$$
Accordingly, we get
\begin{equation}\label{ux}
p=-D_vL \left(x,\frac{j}{m},m \right).
\end{equation}
Since $u_x=p$, we have the first statement. 

Now, using \eqref{ux} in the HJ equation in \eqref{MFGgeneral}, we obtain \eqref{m}, which determines $m$ as a function of $x$ and $j$.
\end{proof}

\begin{proposition}\label{Propc01c10}
Let $m(x,j)$ be determined by \eqref{m}. If $j>0$, then, \eqref{cost} becomes
\begin{equation}\label{c01}
c_{01}(j)=\int_0^1 D_vL \left(x,\frac{j}{m(x,j)},m(x,j)\right).
\end{equation}
Similarly, if $j<0$, then, \eqref{cost10} becomes
\begin{equation}\label{c10}
c_{10}(j)=-\int_0^1 D_vL \left(x,\frac{j}{m(x,j)},m(x,j)\right).
\end{equation}
\end{proposition}

\begin{proof}
For $j>0$, we use \eqref{uxform}, to substitute $u_x$ in \eqref{c01ux}, to get \eqref{c01}. 
The proof of the case $j<0$, is similar.
\end{proof}

\begin{remark}
Note that, \eqref{c01} and \eqref{c10} are not valid, in general, for $j<0$ and $j>0$, respectively. In these cases, we must use \eqref{cost} and \eqref{cost10}. 
\end{remark}


\section{Mean-field game model on a network}
\label{MFG model}

Now, we examine the MFG formulation on networks.
While this problem shares various aspects with the Wardrop model, a  significant difference  is that MFGs are set up in undirected networks. 

\subsection{The network and the data}
In the MFG model, we are given the following.

\begin{enumerate}

\item A finite \textbf{undirected} network, $\Gamma=(E,V)$, where
$E = \{e_k : k \in \{1,2,\dots,n\} \}$ is the set of edges
and $V=\{v_i : i \in \{1,2,\dots, m\} \}$ is the set of vertices.
To any edge $e_k$, we associate the pair $(v_r,v_i)$ of its endpoints.

\item Agents enter the network through  $\lambda$ {\em entrance vertices} and exit it through $\mu$ {\em exit vertices} (disjoint from the entrance vertices).
For convenience, we assume that the last $\mu$ vertices in $V$ are the exit vertices. Furthermore, we suppose that entrance and exit vertices have incidence $1$. If this is not the case, we proceed as in Remark \ref{inoutedges}.

\item A flow of agents, the entry current ${\bm{\iota}}=({\iota}_1,\dots,{\iota}_{\lambda}) >0$, is prescribed at the entrance vertices.
The entry currents in the other vertices are zero.

\item At the $\mu$ exit vertices, agents pay an exit cost $\bm{\phi}$. Here, we assume that this exit cost vanishes. 
If the exit cost is nonzero, we proceed as in Remark \ref{exitremark} by adding an auxiliary edge.

\end{enumerate}

\begin{IMG}
{"2in2out1", 
Graph[{Labeled[1 \[UndirectedEdge] 2, Subscript["e", 1]],
  Labeled[2 \[UndirectedEdge] 3, Subscript["e", 2]],
  Labeled[3 \[UndirectedEdge] 4, Subscript["e", 3]]},
 VertexLabels -> {1 -> Subscript["v", 1], 2 -> Subscript["v", 2], 
   3 -> Subscript["v", 3], 4 -> Subscript["v", 4]}, 
 GraphLayout -> "SpringEmbedding"]
(*\]\]\]\]\]\]\]\]\]*)}	
\end{IMG}

\begin{figure}[ht]
\centering
\includegraphics[scale=1]{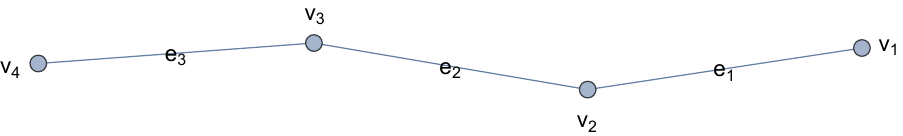}
\caption{MFG network.}
\label{2in2out1}
\end{figure}

\begin{example}\label{2in2outex}
Consider a three edges network like the one in Figure \ref{2in2out1}. Let $v_1$ and $v_3$ be entrance vertices, and let $v_2$ and $v_4$ be exit vertices. 
We attach an exit edge to $v_2$ and an entrance edge to $v_3$, so we get the new network in Figure \ref{2in2out2}.

\end{example}

\begin{IMG}
{"2in2out2", 
Graph[{Labeled[1 \[UndirectedEdge] 2, Subscript["e", 1]],
  Labeled[2 \[UndirectedEdge] 3, Subscript["e", 2]],
  Labeled[3 \[UndirectedEdge] 4, Subscript["e", 3]],
  Labeled[2 \[UndirectedEdge] 5, Subscript["e", 4]],
  Labeled[6 \[UndirectedEdge] 3, Subscript["e", 5]]
  },
 VertexLabels -> {1 -> Subscript["v", 1], 2 -> Subscript["v", 2], 
   3 -> Subscript["v", 3], 4 -> Subscript["v", 4], 
   5 -> Subscript["v", 5], 6 -> Subscript["v", 6] },
 EdgeStyle -> {2 \[UndirectedEdge] 5 -> Dashed, 
   6 \[UndirectedEdge] 3 -> Dashed},
  GraphLayout -> "SpringEmbedding"]
(*\]\]\]\]\]\]\]\]\]*)}	
\end{IMG}

\begin{figure}[ht]
\centering
\includegraphics[scale=1]{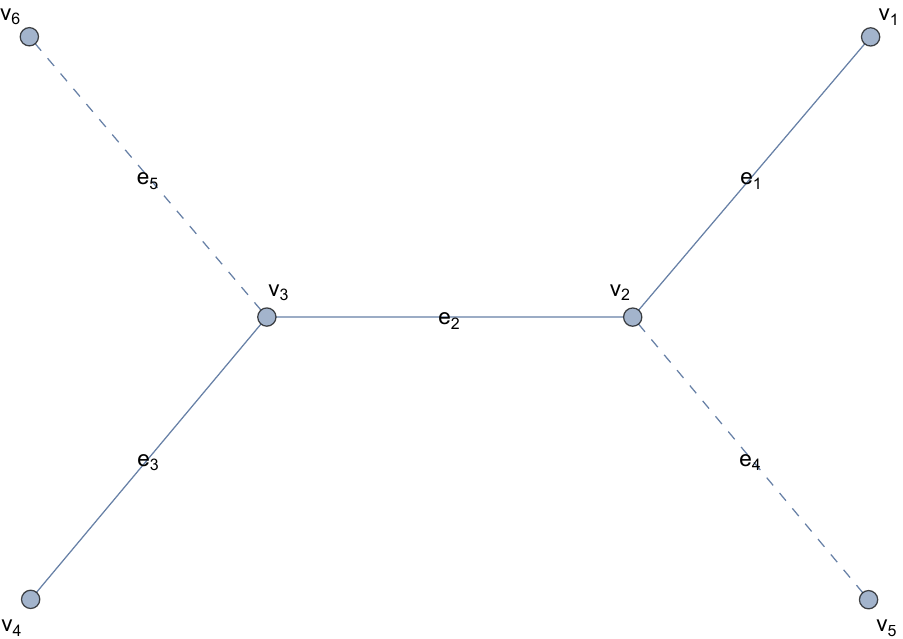}
\caption{MFG network with entrance and exit edges.}
\label{2in2out2}
\end{figure}

\subsection{The variables and the costs}

On the network $\Gamma$, we define the following variables:

\begin{itemize}

\item \textbf{Edge currents:} In each edge, we have a \emph{current} variable representing  the number of agents crossing that edge per unit of time. 
The current $j_k$ in the edge $e_k=(v_r,v_i)$ is decomposed into positive and negative parts; that is,  $j_k= j_k^i-j_k^r$ with
\begin{equation}\label{js}
j_k^i \cdot j^r_k=0, \ \ j_k^i, j^r_k \geq 0,
\end{equation}
where $j^i_k$ is the current going to the vertex $v_i$ and $j^r_k $ is the current going to the vertex $v_r$.
The complementary condition means that all agents in an edge move in the same direction.


\item \textbf{Transition currents:}
The \emph {transition current} from $e_k$ to $e_l$ through a common vertex $v_r$ is denoted by $j_{kl}^r$.
The transition currents also satisfy the complementary condition
\begin{equation}\label{jts}
j_{kl}^r \cdot j_{lk}^r=0, \ \ j_{kl}^r, j_{lk}^r \geq 0.
\end{equation}

\item \textbf{Hamiltonian:}
On each edge $e_k=(v_r,v_i)$, identified with the $[0,1]$ interval, we have a MFG system
\begin{equation}\label{CHkJ}
  \begin{cases}
    H_k(x,u_x,m)=0\\
    -mD_pH_k(x,u_x,m)=j_k,
  \end{cases}
\end{equation}
where $j_k$ is the current on $e_k$, parametrized accordingly.

\item \textbf{Value function:}
The \emph{value function} in the edge $e_k=(v_r,v_i)$ at the vertex $v_i$ is denoted by $u_k^i$.
Note that the notation of the value function emphasizes that it is not defined at the vertex but rather on an edge/vertex pair. 
These values, $u_k^i$ and $u_k^r$, are the boundary data for the HJ equation on the edge $e_k$.
\end{itemize}

On $\Gamma$, we have the following costs

\begin{enumerate}

\item \textbf{Travel costs:}  According to Section \ref{costsection}, in the edge $e_k=(v_r,v_i)$ we define the cost of moving from $v_r$ to $v_i$, $c_k^i(j_k)$, and the cost of moving from $v_i$ to $v_r$, $c_k^r(j_k)$, associated with \eqref{CHkJ}. More precisely, let $(u^{k},m^k)=(u^{k}(x),m^k(x,j_k))$ solves \eqref{CHkJ}, then, $c_k^i(j_k)$ is determined by \eqref{Lv} with the Legendre transform of  $H_k$, $L_k$, and $m^k$. Similarly, $c_k^r(j_k)$ is determined by \eqref{Lvb}  with $L_k$ and $m^k$.

Note that  the cost in the edge depends only on the current in this edge.
The travel cost from an auxiliary entrance vertex to an original vertex in the entrance edge is zero, and the travel cost in the opposite direction is $+\infty$. 
The travel cost from an original vertex to an auxiliary exit vertex is the exit cost, and the travel cost in the opposite direction is $+\infty$. 



\item \textbf{Switching costs:}
For moving from $e_k$ to $e_l$ through a common vertex $v_i$, agents pay a \emph{switching cost}, $\psi_{kl}^i$.
Often, we may require $\psi_{kl}^i \geq 0$.
We assume, for simplicity, that  $\psi_{kl}^i$ is independent of the transition current. 

For any vertex, $v_i$, that has over two incident edges, we require the following triangle-type inequality in the switching costs,
\begin{equation}\label{tri}
\psi_{kl}^i \leq \psi_{kp}^i + \psi_{pl}^i.
\end{equation}
\end{enumerate}

\begin{remark}
If we need to add extra vertices or edges, we modify the switching costs as follows.   
The switching cost from an entrance edge to an original edge is $0$, and from an original edge to an entry edge is $\infty$.
Similarly, the switching cost from an exit edge to an original edge is $\infty$, and from an original edge to an exit edge is $0$.
This avoids entry through exits and exit through entrances.

\end{remark}

\subsection{The equations}\label{equations}

Finally, we set up the equations that determine the MFG. These comprise the MFG system in the edges
and the optimality conditions at the vertices that arise from switching and from the MFG on each edge, as considered in Section \ref{costsection}. 
Further, we have a further condition that represents the balance of the flow of agents at the different vertices, Kirchhoff's conditions.

\paragraph*{\textbf{Optimality conditions at the vertices.}}
In the process of minimizing their travel cost, agents choose the least expensive path.
In particular, they can switch from $e_k$ to $e_l$ through the common vertex $v_i$ by paying a cost $\psi_{kl}^i$. 
This possibility is encoded in the following inequality
 
\begin{equation}\label{optcond}
u^r_k \leq   u_l^r + \psi_{kl}^i, \ \forall \ i,k,l. 
\end{equation}

\paragraph*{\textbf{Complementarity conditions at the vertices.}}
If the transition current $j_{kl}^i$ is not zero, then agents are moving from $e_k$ to $e_l$ through $v_i$.
This implies that $u_k^r=u_l^r+ \psi_{kl}^i$.
So, we have the following
\begin{equation}\label{ujcomp}
j_{kl}^i \cdot (u^r_k- u^r_l -\psi_{kl}^i)=0, \ \ \forall i, k, l.
\end{equation}

\paragraph*{\textbf{Optimality conditions in the edges.}}
In the edge $e_k=(v_r,v_i)$, for any $j_k$ we have
\begin{equation}\label{optcond2}
\begin{cases}
u^r_k \leq c_k^i(j_k) + u_k^i, \\
u^i_k \leq c_k^r(j_k) + u^r_k.
\end{cases}
\end{equation}
Moreover,
$$
\begin{cases}
j_k>0 \implies
u^r_k = c_k^i(j_k^i) + u_k^i, \\
 j_k<0 \implies
u_k^i= c_k^r(-j_k^r) + u^r_k.
\end{cases}
$$
The costs $c_k^i$ and $c_k^r$ may not be the same, unless the Hamiltonian is even, as discussed in Section \ref{costsection}. 

\paragraph*{\textbf{Complementarity conditions in the edges.}}
In the edge $e_k=(v_r,v_i)$, we have the complementary conditions
\begin{equation}\label{compedge}
\begin{cases}
j_k^i \cdot (u_k^r-u_k^i-c_k^i(j_k))=0, \\
j_k^r \cdot \left(u_k^i-u_k^r-c_k^r(j_k)\right)=0.
\end{cases}
\end{equation}

\paragraph*{\textbf{Balance equations and Kirchhoff's law.}}

Consider an edge $e_k=(v_r,v_i)$, and the set $\mathcal{E}$ of the incident edges at $v_i$ distinct from $e_k$. 
The current $j_k^i$ is equal to the sum of the transition currents from $e_k$ to all the other incident edges at $v_i$
\begin{equation}\label{balance1}
\sum_{e_l \in \mathcal{E}} j_{kl}^i= j^i_k;
\end{equation}
this identity models the splitting of the current at a vertex. 
We have a similar equation for the gathering of the transition currents; the current $j_k^r$ is the sum of the transition currents to $e_k$ from all the incident edges at $v_i$
\begin{equation}\label{balance2}
\sum_{e_l \in \mathcal{E}} j_{lk}^i=j^r_k. 
\end{equation}
In particular, 
we have Kirchhoff's law at $v_i$; that is,  the sum of the incoming currents $j_k$ is equal to the sum of the outgoing currents, $j_l$,
\begin{equation}\label{Kirch}
\sum_k j_k=\sum_l j_l,
\end{equation}
because of the current decomposition
$$
\sum_k (j_k^i-j_k^r)= \sum_l (j_l^r-j_l^s).
$$

\paragraph*{\textbf{Entry edge equations.}}
In every entry edge $e_k=(v_r,v_i)$, given the entering current ${\iota}_i$ at the entrance vertex $v_i$, we have
\begin{equation}\label{incondj}
\begin{cases}
j^i_k={\iota}_i,\\
j^r_k=0.
\end{cases}
\end{equation}

\paragraph*{\textbf{Exit edge equations.}}
In every exit edge $e_l=(v_i,v_s)$, we assume the exit cost vanishes, so we have 
\begin{equation}\label{outcondu}
u^s_l \leq 0,
\end{equation}
with equality if $j^s_l>0$.  
In $e_l$, we also have
\begin{equation}\label{outcondj}
j^i_l=0.
\end{equation}




\section{Reformulation of the MFG as Wardrop model}\label{MFGtoW}

Here, we reformulate the MFG model from Section \ref{MFG model} as a Wardrop equilibrium model, as described in Section \ref{Wmodel}.
Then, after identifying the network, the currents, and the costs in the Wardrop model, we show that the MFG solution corresponds to a Wardrop equilibrium.

The network in the Wardrop model is directed. In contrast, in the MFG model, it is undirected. 
To establish the correspondence between MFG and Wardrop models, we show how to build a new directed network, $\bar{\Gamma}$, from the MFG undirected network, $\Gamma$. 

\begin{enumerate}

\item  To each current and transition current, 
 we associate a directed edge. This amounts to doubling all undirected edges and creating transition edges corresponding to the transition currents. 

\item Each of the two vertices of these directed edges corresponds to a pair $(e_k,v_i)$, where $e_k$ is an edge and $v_i$ is one of its vertices in $\Gamma$. This pair is built as follows.

\begin{itemize}

\item In the case of a current $j_k^i$ in an edge $e_k \in \Gamma$ with vertices $(v_r,v_i)$, the vertices of the new edge correspond to the pairs $(e_k,v_r)$ and $(e_k,v_i)$. 

\item In the case of a transition current $j_{kl}^i$ from the edge $e_k$ to the edge $e_l$ through the vertex $v_i$, the new vertices correspond to the pairs $(e_k,v_i)$ and $(e_l,v_i)$. 

\end{itemize}

\begin{remark}
A natural notation for the edges in $\bar{\Gamma}$ is to use the same indices from the corresponding current in the MFG. For example, $\bar{e}_k^i$ for the edge carrying the current $j_k^i$, and  $\bar{e}_{kl}^i$ for $j_{kl}^i$. However, to avoid heavy notation and keep consistency with the previous notation, 
we relabel the edges and the vertices and write $\bar{\Gamma}=(\bar{E},\bar{V})$ for $\bar{E}=\{\bar{e}_{\kappa}: \kappa \in \{ 1,2,\dots,\bar{n} \}\}$ and $\bar{V}=\{\bar{v}_i: i \in \{ 1,2,\dots, \bar{m} \}\}$. When the precise correspondence is needed, we use the identification $\bar{e}_{\kappa}=\bar{e}_k^i$ or $\bar{e}_{\kappa}=\bar{e}_{kl}^i$. 
\end{remark}

\item Because entrance currents at exit vertices and exit currents at entrance vertices vanish, we delete the corresponding edges so that $\bar{\Gamma}$ has entrance and exit vertices with incidence $1$.

\end{enumerate}

\begin{example}

Consider the network in Example \ref{2in2outex}.
We follow the previous steps to transform this network from the MFG to the Wardrop setting. Accordingly, 
we get the new network $\bar{\Gamma}$ in Figure \ref{2in2out3}, which consists of $22$ edges and $10$ vertices. The blue edges correspond to the currents and the red edges correspond to the transition currents. The dashed edges will be deleted, as discussed in the third step above. Thus, these dashed edges are not a part of $\bar{\Gamma}$. 
\end{example}

\begin{IMG}
{"2in2out3", 
Graph[{Labeled[1 \[DirectedEdge] 2, Subscript[OverBar["e"], 1]],
  Labeled[2 \[DirectedEdge] 1, Subscript[OverBar["e"], 19]],
  Labeled[7 \[DirectedEdge] 3, Subscript[OverBar["e"], 2]],
  Labeled[3 \[DirectedEdge] 7, Subscript[OverBar["e"], 6]],
  Labeled[5 \[DirectedEdge] 8, Subscript[OverBar["e"], 21]],
  Labeled[8 \[DirectedEdge] 5, Subscript[OverBar["e"], 4]],
  Labeled[6 \[DirectedEdge] 9, Subscript[OverBar["e"], 5]],
  Labeled[9 \[DirectedEdge] 6, Subscript[OverBar["e"], 22]],
  Labeled[2 \[DirectedEdge] 7, Subscript[OverBar["e"], 7]],
  Labeled[7 \[DirectedEdge] 2, Subscript[OverBar["e"], 8]],
  Labeled[7 \[DirectedEdge] 8, Subscript[OverBar["e"], 9]],
  Labeled[8 \[DirectedEdge] 7, Subscript[OverBar["e"], 10]],
  Labeled[2 \[DirectedEdge] 8, Subscript[OverBar["e"], 11]],
  Labeled[8 \[DirectedEdge] 2, Subscript[OverBar["e"], 12]], 
  Labeled[9 \[DirectedEdge] 3, Subscript[OverBar["e"], 13]], 
  Labeled[3 \[DirectedEdge] 9, Subscript[OverBar["e"], 14]],
  Labeled[3 \[DirectedEdge] 10, Subscript[OverBar["e"], 15]], 
  Labeled[10 \[DirectedEdge] 3, Subscript[OverBar["e"], 16]], 
  Labeled[10 \[DirectedEdge] 4, Subscript[OverBar["e"], 3]], 
  Labeled[4 \[DirectedEdge] 10, Subscript[OverBar["e"], 20]],
  Labeled[9 \[DirectedEdge] 10, Subscript[OverBar["e"], 17]], 
  Labeled[10 \[DirectedEdge] 9, Subscript[OverBar["e"], 18]]
  },
 VertexLabels -> {1 -> Subscript[OverBar["v"], 1], 
   2 -> Subscript[OverBar["v"], 2], 3 -> Subscript[OverBar["v"], 3], 
   4 -> Subscript[OverBar["v"], 4], 5 -> Subscript[OverBar["v"], 5] ,
   6 -> Subscript[OverBar["v"], 6],
   7 -> Subscript[OverBar["v"], 7], 8 -> Subscript[OverBar["v"], 8],
   9 -> Subscript[OverBar["v"], 9], 10 -> Subscript[OverBar["v"], 10]},
 EdgeStyle -> {2 \[DirectedEdge] 1 -> Dashed, 
   4 \[DirectedEdge] 10 -> Dashed, 9 \[DirectedEdge] 6 -> Dashed, 
   5 \[DirectedEdge] 8 -> Dashed,
   2 \[DirectedEdge] 7 -> Red, 7 \[DirectedEdge] 2 -> Red,
   8 \[DirectedEdge] 7 -> Red, 7 \[DirectedEdge] 8 -> Red, 
   2 \[DirectedEdge] 8 -> Red, 8 \[DirectedEdge] 2 -> Red, 
   3 \[DirectedEdge] 9 -> Red, 9 \[DirectedEdge] 3 -> Red, 
   9 \[DirectedEdge] 10 -> Red, 10 \[DirectedEdge] 9 -> Red, 
   3 \[DirectedEdge] 10 -> Red, 10 \[DirectedEdge] 3 -> Red} ,
 GraphLayout -> "SpringEmbedding"]
(*\]\]\]\]\]\]\]\]\]*)}	
\end{IMG}

\begin{figure}[ht]
\centering
\includegraphics[scale=1]{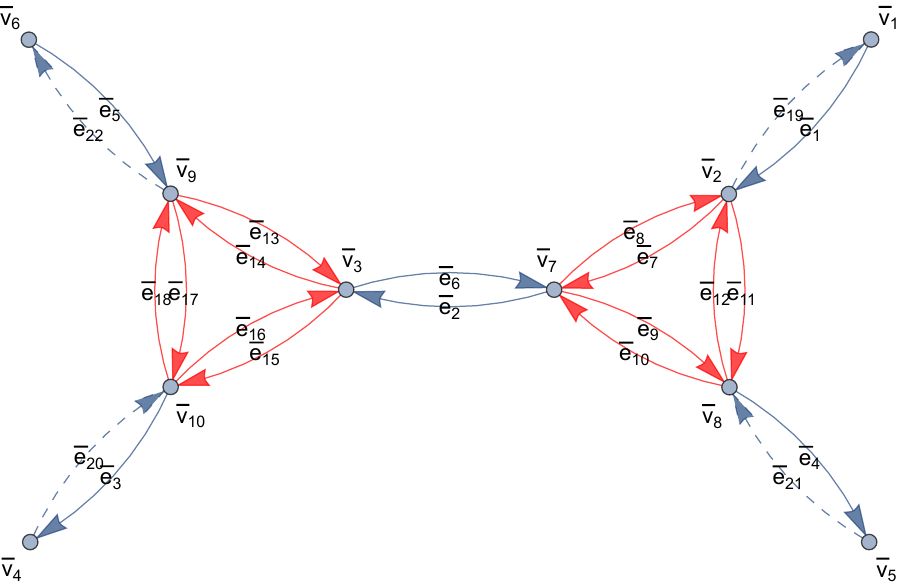}
\caption{New Wardrop network $\bar{\Gamma}$.}
\label{2in2out3}
\end{figure}

The correspondence between the variables in the MFG model and the Wardrop model is as follows.

\begin{enumerate}

\item The value function from the MFG is a function on the new vertices because $u_k^i$ is a function on pairs $(e_k,v_i)$, where $v_i$ is a vertex of $e_k$. 
While this function does not appear explicitly in the  Wardrop model, it is relevant in our analysis.

\item Each current or transition current in the MFG network corresponds to an edge in the new network $\bar{\Gamma}$. Thus, we have a natural correspondence between currents in the two models, $j$ and $\bar{\jmath}$. If $\bar{e}_{\kappa}=\bar{e}_k^i$, we set  $\bar{\jmath}_{\kappa}=j_k^i$, and if $\bar{e}_{\kappa}=\bar{e}_{kl}^i$, we set  $\bar{\jmath}_{\kappa}=j_{kl}^i$.

\item By combining \eqref{balance1} and \eqref{balance2} at every pair $(e_k,v_i)$ which is not an entrance or an exit, we obtain the Kirchhoff's law in the corresponding vertex of $\bar{\Gamma}$

\begin{equation}\label{Kirchhoff}
j_k^r+\sum_{e_l \in \mathcal{E}} j_{kl}^i = j_k^i+ \sum_{e_l \in \mathcal{E}} j_{lk}^i,
\end{equation}
which can be written in terms of the variable $\bar{\jmath}$ straightforwardly.

\item  Moreover, at entrance vertices, where we have the entry current assignments, we have the first equation in \eqref{incondj}. This equation and \eqref{Kirchhoff} are encoded as a matrix equation
$$
K \bar{\bm{\jmath}}=B. 
$$
Thus, we have a linear equation for every vertex in $\bar{\Gamma}$ that is not an exit.

\item 
Let $\bar{e}_k^i \in \bar{\Gamma}$ and $\bar{e}_k^r \in  \bar{\Gamma}$ be the edges corresponding to the undirected edge $e_k=(v_r,v_i) \in \Gamma$, with orientations corresponding to the currents $j_k^i$ and $j_k^r$, respectively.  
Based on the discussion in Section \ref{costsection}, the MFG has two costs $c_k^i$ and $c_k^r$ in the edge $e_k$.
The cost for traveling in $\bar{e}_k^i$ is $\bar{c}_k^i(\bm{j_k})=c_k^i(j_k^i-j_k^r)$, where $\bm{j_k}=(j_k^i,j_k^r)$,
and the cost for traveling in $\bar{e}_k^r$ is $\bar{c}_k^r(\bm{j_k})=c_k^r(j_k^r-j_k^i)$.

\item The switching costs $\psi_{kl}^i$ from the MFG are constant travel costs in the new Wardrop model; that is,  $\bar{c}_{\kappa}(\bm{j}):=\bar{c}^i_{kl}(\bm{j})=\psi_{kl}^i$ in the transition edges $\bar{e}_{\kappa}=\bar{e}_{kl}^i$.


Applying this to the network in Figure \ref{2in2out3}, we get the new Wardrop network as in Figure \ref{2in2out4}.

\begin{IMG}
{"2in2out4", 
Graph[{Labeled[1 \[DirectedEdge] 2, Subscript[OverBar["c"], 1]],
  Labeled[7 \[DirectedEdge] 3, Subscript[OverBar["c"], 2]],
  Labeled[3 \[DirectedEdge] 7, Subscript[OverBar["c"], 6]],
  Labeled[8 \[DirectedEdge] 5, 0],
  Labeled[6 \[DirectedEdge] 9, 0],
  Labeled[2 \[DirectedEdge] 7, Subscript["\[Psi]", 12]],
  Labeled[7 \[DirectedEdge] 2, Subscript["\[Psi]", 21]],
  Labeled[7 \[DirectedEdge] 8, 0],
  Labeled[8 \[DirectedEdge] 7, \[Infinity]],
  Labeled[2 \[DirectedEdge] 8, 0],
  Labeled[8 \[DirectedEdge] 2, \[Infinity]], 
  Labeled[9 \[DirectedEdge] 3, 0], 
  Labeled[3 \[DirectedEdge] 9, \[Infinity]],
  Labeled[3 \[DirectedEdge] 10, Subscript["\[Psi]", 23]], 
  Labeled[10 \[DirectedEdge] 3, Subscript["\[Psi]", 32]], 
  Labeled[10 \[DirectedEdge] 4, Subscript[OverBar["c"], 3]], 
  Labeled[9 \[DirectedEdge] 10, 0], 
  Labeled[10 \[DirectedEdge] 9, \[Infinity]]
  },
 VertexLabels -> {1 -> Subscript[OverBar["v"], 1], 
   2 -> Subscript[OverBar["v"], 2], 3 -> Subscript[OverBar["v"], 3], 
   4 -> Subscript[OverBar["v"], 4], 5 -> Subscript[OverBar["v"], 5] ,
   6 -> Subscript[OverBar["v"], 6],
   7 -> Subscript[OverBar["v"], 7], 8 -> Subscript[OverBar["v"], 8],
   9 -> Subscript[OverBar["v"], 9], 10 -> Subscript[OverBar["v"], 10]},
 EdgeStyle -> {
   2 \[DirectedEdge] 7 -> Red, 7 \[DirectedEdge] 2 -> Red,
   8 \[DirectedEdge] 7 -> Red, 7 \[DirectedEdge] 8 -> Red, 
   2 \[DirectedEdge] 8 -> Red, 8 \[DirectedEdge] 2 -> Red, 
   3 \[DirectedEdge] 9 -> Red, 9 \[DirectedEdge] 3 -> Red, 
   9 \[DirectedEdge] 10 -> Red, 10 \[DirectedEdge] 9 -> Red, 
   3 \[DirectedEdge] 10 -> Red, 10 \[DirectedEdge] 3 -> Red} ,
 GraphLayout -> "SpringEmbedding"]
(*\]\]\]\]\]\]\]\]\]*)}	
\end{IMG}

\begin{figure}[ht]
\centering
\includegraphics[scale=1]{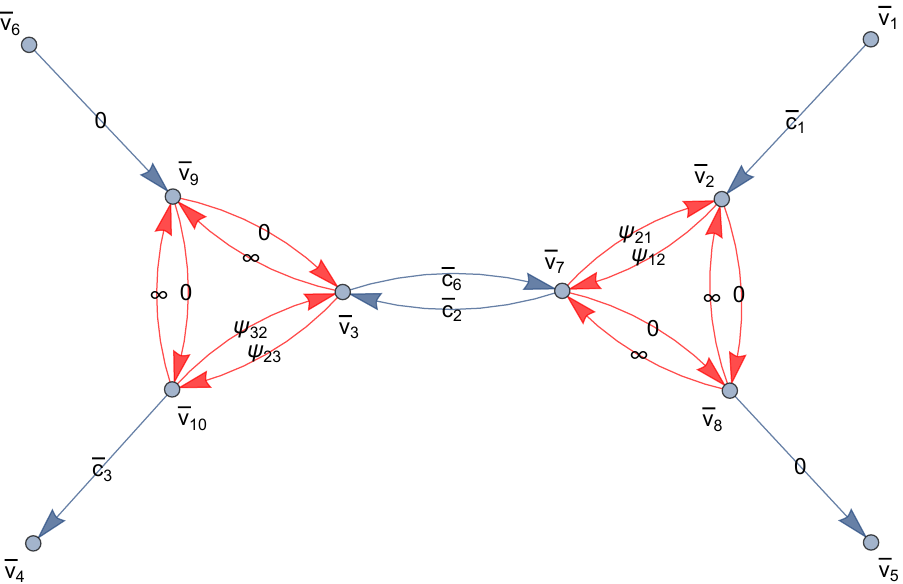}
\caption{New Wardrop network with costs.}
\label{2in2out4}
\end{figure}

\end{enumerate}

The main result connecting MFGs with Wardrop models is

\begin{theorem}\label{soltheorem}

Let M be a MFG model and W be the corresponding Wardrop model. Suppose that a pair $(\bm{u},\bm{{\jmath}^*})$ solves M. Then, the corresponding $\bm{\bar{{\jmath}}^*}$ is a Wardrop equilibrium for W. 
\end{theorem}

\begin{proof}
For any admissible $\bm{\bar{\jmath}}$, we have
$$
K \bm{{\bar{\jmath}}^*}=B \ \text{and} \ K \bm{\bar{\jmath}}=B.
$$
Subtracting the two equations, we get
$$\label{kjj}
K(\bm{{\bar{\jmath}}^*} - \bm{\bar{\jmath}})=0.
$$
Let $\bm{u}$ be the vector of value functions arising from the MFG. 
Multiplying \eqref{kjj} by $\bm{u^{T}}$, gives
$$
\bm{u^{T}}K(\bm{{\bar{\jmath}}^*} - \bm{\bar{\jmath}})=0.
$$
Except for the currents pointing towards an exit vertex (which appears only once), each current appears twice, one for each of its vertices but with different signs.
The previous computation can be organized as 
\begin{equation}\label{lll}
\sum_{\bar{e}_k^i} (u_k^r-u_k^i)(\bar{\jmath}^{i*}_k-\bar{\jmath}_k^i)
+\sum_{\bar{e}_{kl}^i} (u_k^i-u_l^i)(\bar{\jmath}^{i*}_{kl}-\bar{\jmath}_{kl}^i)=0,
\end{equation}
 the first sum is over the edges of W where the vertices of $e_k$ are $v_r$ and $v_i$; this includes the exit edges since, in that case, $u_k^i\leq 0$. The second sum is over the transition edges of W  corresponding to transitions through the vertex $v_i$ from $e_k$ to $e_l$. 
Based on the complementarity conditions in the edges \eqref{compedge}, 
there are two cases for the first term, either 
$$
\bar{\jmath}^{i*}_k > 0 \Rightarrow u_k^r-u_k^i=\bar{c}_k^i(\bm{\bar{\jmath}}^*)
$$
or, noting \eqref{optcond2}, 
$$
\bar{\jmath}^{i*}_k =0 \Rightarrow u_k^r-u_k^i \leq \bar{c}_k^i(\bm{\bar{\jmath}}^*).
$$
Because $\bar{\jmath}^{i}_k \geq 0$,
$$
\bar{\jmath}^{i*}_k=0 \Rightarrow  (u_k^r-u_k^i) (\bar{\jmath}^{i*}_k-\bar{\jmath}_k^i) \geq \bar{c}_k^i(\bm{\bar{\jmath}}^*)(\bar{\jmath}^{i*}_k-\bar{\jmath}_k^i).
$$
Similarly, based on the complementarity conditions at the vertices \eqref{ujcomp}, we have two cases for the second term, either 
$$
\bar{\jmath}^{i*}_{kl} > 0 \Rightarrow u_k^i-u_l^i
=\psi_{kl}
$$
or, noting \eqref{optcond}, 
$$
\bar{\jmath}^{i*}_{kl}=0 \Rightarrow u_k^i-u_l^i
\leq \psi_{kl}.
$$
Because $\bar{\jmath}^{i}_{kl} \geq 0$, 
$$
\bar{\jmath}^{i*}_{kl}=0 \Rightarrow (u_k^i-u_l^i) (\bar{\jmath}^{i*}_{kl}-\bar{\jmath}_{kl}^i) \geq 
\psi_{kl}(\bar{\jmath}^{i*}_{kl}-\bar{\jmath}_{kl}^i).
$$
Using these results in \eqref{lll}, we get

$$
 \sum_{\substack{\bar{e}_k^i \\ \bar{\jmath}^{i*}_k >0}} \bar{c}_k^i(\bm{\bar{\jmath}}^*)(\bar{\jmath}^{i*}_k-\bar{\jmath}_k^i)
  + \sum_{\substack{\bar{e}_k^i \\ \bar{\jmath}^{i*}_k =0}}  \bar{c}_k^i(\bm{\bar{\jmath}}^*)(\bar{\jmath}^{i*}_k-\bar{\jmath}_k^i)
  +\sum_{\substack{\bar{e}_{kl}^i \\ \bar{\jmath}^{i*}_{kl} > 0}}  \psi_{kl}(\bar{\jmath}^{i*}_{kl}-\bar{\jmath}_{kl}^i)
  +\sum_{\substack{\bar{e}_{kl}^i \\ \bar{\jmath}^{i*}_{kl}=0}}  \psi_{kl}(\bar{\jmath}^{i*}_{kl}-\bar{\jmath}_{kl}^i)\leq 0.
$$
This implies
$$
\langle \bm{\bar{c}}( \bm{{\bar{\jmath}}^*}), \bm{{\bar{\jmath}}^*} - \bm{\bar{\jmath}} \rangle \leq 0.
$$
Accordingly,  $\bm{{\bar{\jmath}}^*}$ is a Wardrop equilibrium.
\end{proof}

Next, we prove a uniqueness result for the new Wardrop equilibrium. For this, we present some definitions. 
We say that the cost is \emph{reversible} if the cost for traveling along an edge from left to right is the same as the cost for traveling from right to left along the same edge, i.e., $c_{01}(j)=c_{10}(j)$.
The cost is \emph{even} if $c_{01}(-j)=c_{01}(j)$.
We discuss these properties in Section \ref{costprop}. 

\begin{proposition}\label{unicurr}
If for every edge $e_k=(v_r,v_i)$ corresponding to an edge in the MFG, the costs are reversible,  even,  increasing for positive current, and satisfying \eqref{poscost}, then in the new Wardrop model, the current is unique in the corresponding edge. 
\end{proposition}

\begin{proof}

For an edge $e_k=(v_r,v_i)$, let $\bar{c}_k^i(\bm{j_k})=c_k^i(j_k^i-j_k^r)=c_k^i(j_k)$ be the cost in the new Wardrop, where $c_k^i$ is the cost in the MFG,
$\bm{j_k}=(j_k^i,j_k^r)$.
Similarly, let $\bar{c}_k^r(\bm{j_k})=c_k^r(j_k^r-j_k^i)=c_k^r(-j_k)$.
For the Wardrop monotonicity, we need monotonicity over all edges. Thus, we need to examine
the expression
$$
\sum_k (\bar{c}_k({\bm{\tilde{\jmath}}})-\bar{c}_k({\bm{\hat{\jmath}}}),
{\bm{\tilde{\jmath}}}-{\bm{\hat{\jmath}}}), 
$$
which can be written for each edge
$$
\sum_k (c_k^i({\tilde{\jmath}}_k)- c_k^i({\hat{\jmath}}_k))
({\tilde{\jmath}}_k^i - {\hat{\jmath}}_k^i ) +
(c_k^r(-{\tilde{\jmath}}_k)- c_k^r(-{\hat{\jmath}}_k))
({\tilde{\jmath}}_k^r - {\hat{\jmath}}_k^r) .
$$
The costs corresponding to the transition currents do not play any role because they are constants.
Since the cost is reversible (i.e.
$c_k^i(j)=c_k^r(j)=c(j)$), we get
\begin{equation}\label{revcost}
\sum_k (c({\tilde{\jmath}}_k)- c({\hat{\jmath}}_k))
({\tilde{\jmath}}_k^i - {\hat{\jmath}}_k^i ) +
(c(-{\tilde{\jmath}}_k)- c(-{\hat{\jmath}}_k))
({\tilde{\jmath}}_k^r - {\hat{\jmath}}_k^r).
\end{equation}
Because of \eqref{WardComp}, we verify four cases for Wardrop equilibrium. These  are as follows.
\begin{enumerate}
\item If ${\tilde{\jmath}}_k^r={\hat{\jmath}}_k^r=0$,
then \eqref{revcost} becomes
$$
\sum_k (c({\tilde{\jmath}}_k^i)-c({\hat{\jmath}}_k^i))({\tilde{\jmath}}_k^i-{\hat{\jmath}}_k^i).
$$
\item If ${\tilde{\jmath}}_k^i={\hat{\jmath}}_k^i=0$, 
then  since $c$ is even, \eqref{revcost} becomes
$$
\sum_k (c(-{\tilde{\jmath}}_k^r)-c(-{\hat{\jmath}}_k^r))({\tilde{\jmath}}_k^r-{\hat{\jmath}}_k^r).
$$
\item If ${\tilde{\jmath}}_k^i={\hat{\jmath}}_k^r=0$, 
then, since $c$ is even, \eqref{revcost} becomes
$$
\sum_k (c({\tilde{\jmath}}_k^r)-c({\hat{\jmath}}_k^i))({\tilde{\jmath}}_k^r-{\hat{\jmath}}_k^i).
$$
\item If ${\tilde{\jmath}}_k^r={\hat{\jmath}}_k^i=0$,
then \eqref{revcost} becomes
$$
\sum_k (c({\tilde{\jmath}}_k^i)-c({\hat{\jmath}}_k^r))({\tilde{\jmath}}_k^i-{\hat{\jmath}}_k^r).
$$
\end{enumerate}
In all four cases, the sums are non-negative because the cost $c(j)$ is increasing for $j>0$.
The inequality is strict if $c(j)$ is strictly increasing and ${\tilde{\jmath}}_k^i \neq {\hat{\jmath}}_k^i$. 
Because the cost is strictly monotone,  by Theorem \eqref{uni}, the Wardrop equilibrium is unique. 
\end{proof}

\begin{remark}
For general costs, a condition that implies uniqueness is the following 
\begin{enumerate}
\item The cost $c_k^i$ is increasing in $\mathbb{R}^+_0$, if ${\tilde{\jmath}}_k^r={\hat{\jmath}}_k^r=0$.

\item The cost $c_k^r$ is increasing in $\mathbb{R}^+_0$, if ${\tilde{\jmath}}_k^i={\hat{\jmath}}_k^i=0$. 

\item The following inequality holds
$$
{\tilde{\jmath}}_k^r c_k^r(-{\hat{\jmath}}_k^i)+{\hat{\jmath}}_k^ic_k^i(-{\tilde{\jmath}}_k^r) \leq {\tilde{\jmath}}_k^r c_k^r({\tilde{\jmath}}_k^r) +{\hat{\jmath}}_k^ic_k^i({\hat{\jmath}}_k^i)
$$
if ${\tilde{\jmath}}_k^i={\hat{\jmath}}_k^r=0$.

\item The following inequality holds
$$
{\tilde{\jmath}}_k^ic_k^i(-{\hat{\jmath}}_k^r)+{\hat{\jmath}}_k^rc_k^r(-{\tilde{\jmath}}_k^i) \leq {\tilde{\jmath}}_k^ic_k^i({\tilde{\jmath}}_k^i)+ {\hat{\jmath}}_k^rc_k^r({\hat{\jmath}}_k^r)
$$
if ${\tilde{\jmath}}_k^r={\hat{\jmath}}_k^i=0$.

\end{enumerate}

\end{remark}

\begin{remark}

In Proposition \ref{unicurr}, we proved the uniqueness of the currents.  However, under the same assumptions, we do not have uniqueness of the transition currents.  For example, 
consider $4$ edges that intersect in a single vertex, as in Figure \ref{cross}.
Let the currents in $e_1$ and $e_3$ be $j_1=j_3=10$, and the currents in $e_2$ and $e_4$ be $j_2=j_4=5$. 
For the transition currents, we may have $j_{13}=10$, $j_{24}=5$, and $j_{14}=j_{23}=0$. Alternatively, we may have, for example, $j_{13}=7.5$ and $j_{24}=j_{14}=j_{23}=2.5$.
\end{remark}

\begin{IMG}
{"cross1", 
Graph[{Labeled[1 \[DirectedEdge] 2, Subscript["e", 1]], 
  Labeled[3 \[DirectedEdge] 2, Subscript["e", 2]], 
  Labeled[2 \[DirectedEdge] 4, Subscript["e", 4]], 
  Labeled[2 \[DirectedEdge] 5, Subscript["e", 3]]}, 
 GraphLayout -> "RadialDrawing"]
(*\]\]\]\]\]\]\]\]\]*)}	
\end{IMG}

\begin{figure}[ht]
\centering
\includegraphics[scale=1]{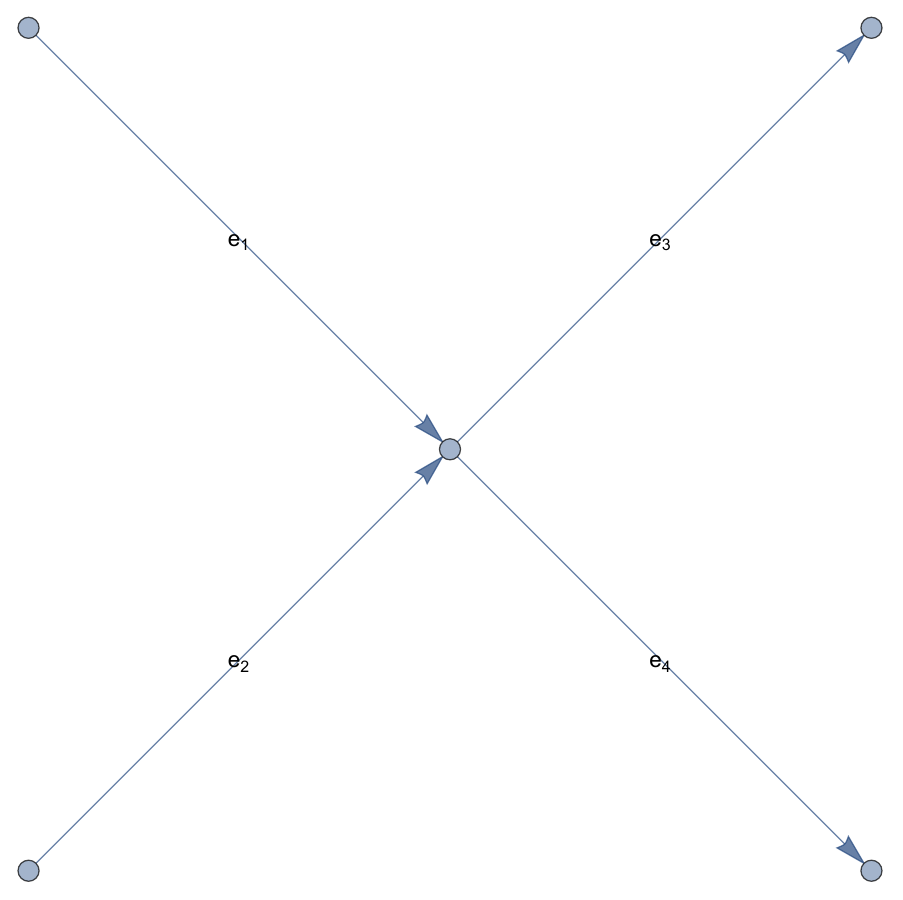}
\caption{4 edges intersect in 1 vertex.}
\label{cross}
\end{figure}

\begin{proposition}\label{pro--unique}
Under the same assumptions of Proposition \ref{unicurr},  in the MFG, the current is uniquely defined in each edge. 
\end{proposition}

\begin{proof}
Suppose the MFG has two current solutions in one edge. 
By theorem \ref{soltheorem}, these two currents are the Wardrop equilibrium current solutions, and by the uniqueness of the currents in the Wardrop model, these two currents are the same. 
\end{proof}


\section{Recovering MFG solution from the Wardrop solution}\label{WtoMFg}

Now, we show how to recover the MFG solution from the corresponding Wardrop model solution.
We start with a MFG problem and build the corresponding Wardrop problem.
We follow the same procedure as in Section \ref{MFGtoW} to convert the MFG model into a Wardrop model. Accordingly, the structure of the network and the costs arise from the MFG.
Then, we solve the new Wardrop problem and use its solution to recover the solution of the MFG. 
This consists in recovering the currents, the transition currents, and the value function.
To recover the current in an edge $e_k$ of a MFG from the current in the Wardrop model, we consider the difference of the currents in the two directed edges associated to $e_k$ in the Wardrop model.
For the transition currents, the correspondence is immediate. 
Kirchhof's law \eqref{Kirchhoff} in the new Wardrop network, $\bar{\Gamma}$, implies the splitting and gathering equations \eqref{balance1} and \eqref{balance2} in the MFG network $\Gamma$, as the following proposition shows. 

\begin{proposition}\label{jprop}
Consider a MFG model satisfying \eqref{poscost} and \eqref{tri} with strict inequality. Let  $\bm{\tilde{{\jmath}}^*}$ be the corresponding Wardrop equilibrium. 
Then, \eqref{balance1} and \eqref{balance2} hold on every non-entrance and non-exit vertex.  
\end{proposition}

\begin{proof}
In the Wardrop model, we have Kirchhoff's law \eqref{Kirchhoff}  for every non-entrance and non-exit vertex. 
Because of \eqref{js}, either $j_k^r=0$ or $j_k^i=0$. 
Without loss of generality, assume that $j_k^i=0$, then \eqref{Kirchhoff} becomes
\begin{equation}\label{Kir2}
j_k^r+\sum_{e_l \in \mathcal{E}} j_{kl}^i = \sum_{e_l \in \mathcal{E}} j_{lk}^i. 
\end{equation}
We must show that the corresponding transition currents $j_{kl}^i$ are zero. 
Suppose, by contradiction, that one of the currents $j_{kl}^i$ is nonzero, then by \eqref{Kir2}, we have $j_{\tilde{l}k}^i>0$, for some $\tilde{l}$.
Consider the Wardrop equilibrium $\bm{\tilde{{\jmath}}^*}$, take $\sigma>0$ small enough, and set $\bm{\hat{\jmath}}$ by changing the following coordinates.  
$$
\hat{\jmath}_{kl}^i= j_{kl}^{i*}-\sigma,
$$
$$
\hat{\jmath}_{\tilde{l}k}^i= j_{\tilde{l}k}^{i*}-\sigma,
$$
and 
$$
\hat{\jmath}_{\tilde{l}l}^i= j_{\tilde{l}l}^{i*}+\sigma.
$$
Note that $\bm{\hat{\jmath}} \in \mathcal{A}$. Because  $\bm{\tilde{{\jmath}}^*}$ is a Wardrop equilibrium, 
 we get 
\begin{equation*}
\sigma(-c_{\tilde{l}l}(\bm{\tilde{{\jmath}}^*})+c_{\tilde{l}k}(\bm{\tilde{{\jmath}}^*})+c_{kl}(\bm{\tilde{{\jmath}}^*})) \leq 0,
\end{equation*}
which is 
\begin{equation*}
\sigma(-\psi^i_{\tilde{l}l}+\psi^i_{\tilde{l}k}+\psi^i_{kl}) \leq 0. 
\end{equation*}
Thus, 
$$
\psi^i_{\tilde{l}l} \geq \psi^i_{\tilde{l}k}+\psi^i_{kl},  
$$
which contradicts the strict inequality  \eqref{tri}. Hence,  $j_{kl}^i$ must be zero. 
Using this in \eqref{Kir2}, we get \eqref{balance2}. Similarly, we get \eqref{balance1}. 
\end{proof}

Next, we retrieve the value function.
For this, we present the following definitions and a key result.
\begin{definition}
Given a Wardrop equilibrium, a \emph{regular vertex} is a vertex that belongs to an edge where the current is positive. 
\end{definition}
\begin{definition}
A walk is a sequence of vertices and edges of a graph. A positively directed walk is a walk of regular vertices connected by edges with positive currents. 
\end{definition}
\begin{lemma}\label{Lemma}
Suppose all loops have positive costs.
Any regular vertex is connected to an exit by a positive walk. Moreover, Wardrop equilibrium admits no loops. 
\end{lemma}
\begin{proof}
Let $\tilde{v}$ be a regular vertex. Then, at $\tilde{v}$, there are walks of positive currents that either end at an exit or have an infinite number of vertices.
Having these loops is incompatible with Wardrop equilibrium.
To prove this, consider a Wardrop equilibrium $\bm{{\tilde{\jmath}}^*}$ 
and a  loop of positive currents. 
Next, take $\sigma >0$ small enough.
For the currents in each edge of the loop, we set
$${\hat{{\jmath}}}_k={{\tilde{\jmath}}^*}_k - \sigma.$$
For all other edges, ${\hat{{\jmath}}}_k={{\tilde{\jmath}}^*}_k$.
Since the cost is positive, we have
\begin{equation}
 \langle \tilde{c}({\bm{\tilde{\jmath}}}^*),{\bm{\tilde{\jmath}}}^*-({\bm{\tilde{\jmath}}}^*- \sigma) \rangle = 
 \sigma \tilde{c}({\bm{\tilde{\jmath}}}^*) \geq 0, 
\end{equation}
which contradicts the Wardrop equilibrium condition.
\end{proof}

Because of Lemma \ref{Lemma}, we can always connect any regular vertex $\tilde{v}_i \in \tilde{e}_k=(\tilde{v}_r,\tilde{v}_i)$ to an exit vertex.
Hence, we define
\begin{equation}\label{val}
\tilde{u}_k^i=\sum_l \tilde{c}_l({\tilde{\jmath}}_l),
\end{equation}
where $\tilde{u}_k^i$ is the candidate for the value function at $\tilde{v}_i$, 
$\tilde{c}_l({\tilde{\jmath}}_l)$ is the cost in each edge of the walk and the sum is taken over the 
current-carrying edges.

\begin{proposition}\label{prop1}
Given the Wardrop equilibrium $\bm{{\tilde{\jmath}}^*}$,  the value function $\tilde{u}$ in \eqref{val} is well defined at all regular vertices, i.e.,
given two different walks starting at the same regular vertex, the value function $\tilde{u}$ in these walks is the same.
Moreover, a walk to an exit is at least as expensive as a current-carrying walk to any exit. 
\end{proposition}
\begin{proof}
Let $\tilde{u}$ be the value function at a regular vertex with two outgoing walks, and suppose the first walk is less expensive than the second one. 	
Consider the Wardrop equilibrium $\bm{{\tilde{\jmath}}^*}$ and take $\sigma >0$ small enough.
For the currents in each edge of the first walk, we set
$${{\jmath}}^1_k={{\tilde{\jmath}}^*}_k+\sigma.$$
For the currents in each edge of the second walk, we set
$${{\jmath}}^2_k={{\tilde{\jmath}}^*}_k-\sigma.$$
For all other edges that are not in the walks or that are common between them, ${{\jmath}}^1_k={{\jmath}}^2_k={{\tilde{\jmath}}^*}_k$.
Since the new corresponding currents $\bm{{{\jmath}^1}}$ and $\bm{{{\jmath}^2}}$ satisfy Kirchhoff's law,
applying the definition of Wardrop equilibrium \eqref{Wardrop} on the second walk, we get
\begin{equation}
\langle \tilde{c}({\bm{\tilde{\jmath}}}^*),{\bm{\tilde{\jmath}}}^*-({\bm{\tilde{\jmath}}}^*- \sigma) \rangle= 
\sigma \tilde{c}({\bm{\tilde{\jmath}}}^*) \geq 0
\end{equation}
because the cost is positive, which contradicts the Wardrop equilibrium condition.

Note that if the second walk is not current-carrying, the proof still holds; following the same steps, we can prove that for any current-carrying walk and any other walk, the other walk is at least as expensive
as the current-carrying walk.
\end{proof}

\begin{proposition}\label{uprop}
Consider a Wardrop model and suppose all vertices are regular. Then, the value function $\tilde{u}$ in the Wardrop model is the value function $u$ in the MFG model and satisfies \eqref{optcond} and \eqref{optcond2}. 
\end{proposition}
\begin{proof}
The value function $u$ is the infimum among all walks.
Hence, by Proposition \ref{prop1}, $\tilde{u}=u$.
By Proposition \ref{prop1}, for $\tilde{e}_k=(\tilde{v}_r,\tilde{v}_i)$, we have 

$$
 \tilde{u}^r_k \leq  \tilde{c_k}(\tilde{\jmath}_k)+\sum_l\tilde{c}_l(\tilde{\jmath}_l), 
$$
The sum here is over the current-carrying edges. 
By \eqref{val}, we can write the previous inequality as 
$$
\tilde{u}^r_k \leq  \tilde{c_k}(\tilde{\jmath}_k)+\tilde{u}^i_k.
$$
If $\tilde{\jmath}_k >0$, then we have 
$$
\tilde{u}^r_k =  \tilde{c_k}(\tilde{\jmath}_k)+\tilde{u}^i_k.
$$

If the edge $\tilde{e}_k$ is a transition edge, then we can get \eqref{optcond}
similarly; we just need to replace the travel costs
by the switching costs.  
\end{proof}

\begin{theorem}\label{theorem-Wardrop-to-MFG}
The MFG equations hold if all vertices are regular in the Wardrop model.	
\end{theorem}
\begin{proof}
In Proposition \ref{jprop}, we proved the correspondence between the currents and the balance equations. 
In Proposition \ref{uprop}, we proved the correspondence between the value functions and the optimality equations. 
Hence, the MFG equations are satisfied. 
\end{proof}
By this, we proved that we can recover the solution of an MFG  problem, $(\bm{u},\bm{{\jmath}^*})$
from the curents of the corresponding Wardrop problem, $\bm{\tilde{{\jmath}}^*}$.


%
%

\section{Cost properties}\label{costprop}

In this section, we study the properties of the cost implied by the microstructure of the MFG in a single edge. 
We begin by addressing reversibility, and then we examine monotonicity.

\subsection{Reversibility of the cost}
Now, we study the reversibility of the costs; that is, whether for a given current, the travel direction of a single agent is the same whether he/she travels against to/with the current. 
In reversible MFG, agents' costs are insensitive to the direction of travel and only depend on the density of other agents.      

\begin{proposition}\label{revcostprop}
If the Hamiltonian is an even function in $p$, i.e., $H(x,-p,m)=H(x,p,m)$, and strictly convex, then we have the following
\begin{enumerate}
\item The cost is reversible; that is, $c_{01}(j)=c_{10}(j), \ \forall j \in \mathbb{R}$.
\item The cost is even; that is, $c_{01}(-j)=c_{01}(j), \ \forall j \in \mathbb{R}$.
\end{enumerate} 
\end{proposition}

\begin{proof}
Since $H$ is even, $L$ is also even, i.e.,
$$
L(x,-v,m)=L(x,v,m).
$$
The density $m$ is determined by solving \eqref{m}.
For $j>0$, traveling from $0$ to $1$, the optimal velocity is $v=\frac{j}{m}$, which solves the necessary optimality condition \eqref{L/v}.
For same $j$ and $m$, we look at the optimal trajectory connecting $1$ to $0$.
We claim that the velocity $-v$ is optimal.
This is true because $-v$ also solves \eqref{L/v}.
Then, by \eqref{cost}, we have
$$c_{01}(j)=c_{10}(j).$$
Thus, the cost is reversible.

Now we prove the second statement. 
Consider the HJ equation in 
 \eqref{MFGgeneral} and  the definition of the current in \eqref{ABC}. 
  If we  replace $j$ by $-j$, $p=u_x$ by $-p=-u_x$, and keep $m$ unchanged, the HJ in \eqref{MFGgeneral} holds because $H$ is even and \eqref{ABC} holds because $D_pH$ is odd in $p$.  
Hence,
\begin{equation}\label{meven}
 m(x,-j)=m(x,j). 
\end{equation}

Because $H$ is strictly convex, $L$ is strictly convex. Thus, by Proposition \ref{Propforv}, there exists $v_+^*>0$ unique solution of \eqref{L/v} in $\mathbb{R}$ . Using \eqref{meven} in \eqref{L/v} we get  

$$ 
 v_+^*(x,-j)=v_+^*(x,j), \ \ \forall j.  
$$ 
Similarly, 
$$
 v_-^*(x,-j)=v_-^*(x,j), \ \ \forall j.  
 $$
Then, \eqref{cost} implies
$$c_{01}(-j)=c_{01}(j).$$
\end{proof}



\begin{example}[Even Hamiltonian]
Suppose the Hamiltonian is
$$
H(x,p,m)=\frac{p^2}{2}-m.
$$
Then, the Lagrangian is
$$
L(x,v,m)=\frac{v^2}{2}+m. 
$$
The corresponding  MFG system is
\begin{equation*}
\begin{cases}
\frac{u_x^2}{2}=m,\\
(-mu_x)_x=0.
\end{cases}
\end{equation*}
The density $m$ is obtained by substituting $u_x=-\frac{j}{m}$ in the HJ equation, which gives the identity 
$$
\frac{1}{2} \left( -\frac{j}{m} \right) ^2=m. 
$$
Accordingly, 
$$
m=\left( \frac{j^2}{2} \right) ^{1/3}.
$$
On the other hand, \eqref{DvL} implies
$$
v=(2m)^{1/2}.
$$
Hence, substituting the formula for $m$ in the previous identity, we get
$$
v=\left( 2 \left( \frac{j^2}{2} \right) ^{1/3} \right) ^{1/2}.
$$
Then, by \eqref{cost}, the cost is
$$
c_{01}(j)=\int_0^1 2^{1/3} |j|^{1/3} dx = 2^{1/3}|j|^{1/3}.
$$
This cost is independent of the sign of $j$ because the Hamiltonian is even.
We can see this using \eqref{c01}, where we get the same cost
$$
c_{01}(j)=\int_0^1 \frac{j}{m} = \frac{j}{(j^2/2)^{1/3}}=(2j)^{1/3}.
$$
\end{example}

\subsection{Monotonicity of the costs}

For general MFGs of the form \eqref{MFGgeneral}, 
under the monotonicity condition 
\begin{gather}\label{matrix}
    \begin{bmatrix}
    -\frac{2}{m}D_mH & D^2_{pm}H \\
    D^2_{pm}H & 2D^2_{pp}H
    \end{bmatrix}\geq 0,
    \end{gather}
the uniqueness of the solution in several cases was proved in \cite{cursolions} (also see \cite{porretta}).
As we see next, this condition is relevant for the monotonicity of the costs. 

\begin{proposition}\label{prop:mon}
For $j>0$, the following holds
\begin{gather}\label{djsol}
\begin{bmatrix}
\frac{\partial c_{01}}{\partial j} \\
\frac{\partial m}{\partial j}
\end{bmatrix}
=
\frac{1}{(D_pH)^2+mD_pHD_{pm}^2H-mD_mHD_{pp}^2H}
\begin{bmatrix}
-D_mH \\
-D_pH
\end{bmatrix}.
\end{gather}
Moreover, if \eqref{matrix} holds, then 
$$
\frac{\partial c_{01}}{\partial j} \geq 0.
$$
For $j<0$, the following holds
\begin{gather}\label{djsol2}
\begin{bmatrix}
\frac{\partial c_{10}}{\partial j} \\
\frac{\partial m}{\partial j}
\end{bmatrix}
=
\frac{1}{-(D_pH)^2+mD_pHD_{pm}^2H+mD_mHD_{pp}^2H}
\begin{bmatrix}
-D_mH \\
D_pH
\end{bmatrix}.
\end{gather}
Moreover, if \eqref{matrix} holds, then 
$$
\frac{\partial c_{10}}{\partial j} \leq 0.
$$
\end{proposition}

\begin{proof}
Assuming $j>0$, and differentiating \eqref{mfgsys} with respect to $j$, we get
\begin{equation}\label{dj}
\begin{cases}
\begin{aligned}
 & -D_pH(x,-c_{01},m) \frac{\partial c_{01}}{\partial j}+D_mH(x,-c_{01},m) \frac{\partial m}{\partial j}=0, \\
& mD^2_{pp}H(x,-c_{01},m)\frac{\partial c_{01}}{\partial j} -\left(D_pH(x,-c_{01},m)+ 
mD^2_{pm}H(x,-c_{01},m)\right)\frac{\partial m}{\partial j}=1,
\end{aligned}
\end{cases}
\end{equation}
where $c_{01}=c_{01}(x,j)$ and $m=m(x,j)$.
The solution of \eqref{dj} is given by \eqref{djsol}, which proves the first result. 

For the second result, notice that $D_mH \leq 0$ and the determinant of the matrix in \eqref{matrix} is
$$
A=-\frac{4}{m}D_mHD^2_{pp}H-(D^2_{pm}H)^2 \geq 0.
$$
Using this and by Cauchy's inequality, we get that the denominator in \eqref{djsol} satisfies 
\begin{equation*}
\begin{aligned}
&M=(D_pH)^2+mD_pHD_{pm}^2H-mD_mHD_{pp}^2H \\
& \geq (D_pH)^2-(D_pH)^2- \frac{1}{4}m^2(D^2_{pm}H)^2-D_mHD^2_{pp}Hm \\
& = \frac{Am^2}{4} \geq 0.
\end{aligned}
\end{equation*}
Consequently,  \eqref{matrix} implies
$$\frac{\partial c_{01}}{\partial j} \geq 0.$$
The proof for the case $j<0$ is similar. 
\end{proof}

\begin{remark}\label{remark:mono}
If $D_mH(x,-c_{01},m) \leq 0$, then $m$ and $c_{01}$ have the same monotonicity with respect to $j$ and opposite otherwise. This is so because we have by the first equation in \eqref{dj} that 
\begin{equation}
\frac{\partial c_{01}}{\partial j}=- \frac{mD_mH(x,-c_{01},m)}{j}\frac{\partial m}{\partial j}.
\end{equation}
\end{remark}

The condition \eqref{matrix} is not necessary for the monotonicity of the cost, as
we show in the following example for separable Hamiltonians.

\begin{example}\label{ex:mono}
Suppose $H$ is separable, that is, of the form $H(x,p,m)-g(m)$. Then 
\[
D^2_{pm}H=0\ \text{and} \ D_mH=-g^{\prime}(m).
\]
From \eqref{djsol}, we get
\begin{equation}\label{dc1}
\frac{\partial c_{01}}{\partial j}=\frac{g^{\prime}(m)}{(D_pH)^2+D_{pp}Hmg^{\prime}(m)}.
\end{equation}
If $g^{\prime}(m) >0$ then $\frac{\partial c_{01}}{\partial j} \geq 0$.
If $g^{\prime}(m) <0$, the analysis is more delicate.

Let $\gamma>1$. Consider the case where
\[
H(p,m)=\frac{|p|^\gamma}{\gamma}-g(m).
\]
Then, we have
\begin{equation*}
\begin{aligned}
& D_pH(p,m) = p|p|^{\gamma -2}, \\
& D_pH(p,m)^2=|p|^{2\gamma -2}\implies (D_pH(-c_{01},m))^2=(\gamma g(m))^{2-\frac{2}{\gamma}}, \\
& D^2_{pp}H(p,m)= (\gamma-1)|p|^{\gamma -2}\implies D^2_{pp}H(-c_{01},m)=
(\gamma-1)(\gamma g(m))^{1-\frac{2}{\gamma}}.
\end{aligned}
\end{equation*}
Substituting in \eqref{dc1}, we get
\begin{equation}
\frac{\partial c_{01}}{\partial j}=\frac{g^{\prime}(m)}{(\gamma g(m))^{2-\frac{2}{\gamma}}+(\gamma-1)(\gamma g(m))^{1-\frac{2}{\gamma}}mg^{\prime}(m)}.
\end{equation}
Thus, to get $\frac{\partial c_{01}}{\partial j} \geq 0$ when $g^{\prime}(m) <0$, we need
\begin{equation}\label{cond}
\gamma g(m) +(\gamma -1)m g^{\prime}(m) \leq 0.
\end{equation}
A case where the prior inequality holds is the following.
Consider
$$g(m)=m^{-\beta}, \ \beta>0.$$
Then, \eqref{cond} becomes
$$\gamma-(\gamma-1)\beta \leq 0.$$
Hence, for
$\beta \geq \frac{\gamma}{\gamma-1}, $
we have $\frac{\partial c_{01}}{\partial j} \geq 0$.
\end{example}

Now, we study the monotonicity of $c_{01}$ for $j \leq 0$.
In this case, agents are traveling against the current, so we cannot use the result in Proposition \ref{Propc01c10} to compute $c_{01}$.

\begin{proposition}\label{prop:cost-jneg}

For $j<0$, the following holds
 \begin{equation}\label{resultjneg}
    \frac{\partial c_{01}(j)}{\partial j} = \int_0^1 {\frac{1}{v^*_+} D_mL(x,v^*_+(x,j),m(x,j))\frac{\partial m(x,j)}{\partial j}} dx.
    \end{equation}
   Moreover, if \eqref{matrix} holds, then 
   
 \begin{equation}\label{sign}
    \frac{\partial c_{01}(j)}{\partial j} \leq 0. 
    \end{equation}
    
\end{proposition}
\begin{proof}
    To simplify the notation, we will write $v$ instead of $v^*_+$. 
    We write \eqref{DvL} as follows
    $$
    v(x,j)D_vL(x,v(x,j),m(x,j))-L(x,v(x,j),m(x,j))=0.
    $$
    Then, we differentiate with respect to $j$ to get
    \begin{equation}\label{DvvL}
        \begin{aligned}
            & v(x,j)D^2_{vv}L(x,v(x,j),m(x,j))\frac{\partial v(x,j)}{\partial j}
            +v(x,j)D^2_{vm}L(x,v(x,j),m(x,j))\frac{\partial m(x,j)}{\partial j} \\
            & -D_mL(x,v(x,j),m(x,j))\frac{\partial m(x,j)}{\partial j}=0.
        \end{aligned}
    \end{equation}
    Also, differentiating \eqref{cost} with respect to $j$, we get
    \begin{equation*}
        \begin{aligned}
            & \frac{\partial c_{01}(j)}{\partial j} = \int_0^1
            D^2_{vv}L(x,v(x,j),m(x,j))\frac{\partial v(x,j)}{\partial j} \\
            & +D^2_{vm}L(x,v(x,j),m(x,j))\frac{\partial m(x,j)}{\partial j} dx.
        \end{aligned}
    \end{equation*}
    Substituting $D^2_{vv}L(x,v(x,j),m(x,j))\frac{\partial v(x,j)}{\partial j}$ from \eqref{DvvL}, we get \eqref{resultjneg}. 
    To prove the second result, notice that in \eqref{resultjneg}, we have 
    $$
    \frac{1}{v^*_+} D_mL(x,v^*_+(x,j),m(x,j)) \geq 0, 
    $$
   and by \eqref{djsol2},  we have 
    $$
    \frac{\partial m(x,j)}{\partial j} \leq 0. 
    $$
    Thus, we get \eqref{sign}. 
\end{proof}


\subsection{Analysis of the velocity field}\label{analysisofv}
Lions' monotonicity condition \eqref{matrix} is related to crowd aversion; that is, agents try to avoid congested areas. 
As we have shown before, in this case, 
both $c$ and $m$ increase with $j$.
It is possible, however, to have a crowd-seeking behavior,  that is, 
$g$ is decreasing, with $c$ increasing in $j$ and $m$ decreasing in $j$.
This is something that happens in uncongested highways.
As speed increases, cars increase their distances. Hence, density decreases.
Here, we would like to investigate the dependence of the optimal velocity field
\begin{equation}\label{optv}
v(x,j)=\frac{j}{m(x,j)}
\end{equation}
on the current and on the density. 
We begin with the dependence on the current. 

\begin{proposition}\label{Prop:v-j}
Let $M=(D_pH)^2+mD_pHD_{pm}^2H-mD_{pp}^2HD_mH$. For $j>0$ 
$$
\frac{\partial}{\partial j}\left[v(x,j)\right]=
- \frac{vD^2_{pm}H+D^2_{pp}HD_mH}{M}.
$$
Moreover, in the separable case, if  Lions condition holds, then 
$$
\frac{\partial}{\partial j}\left[v(x,j)\right] \leq 0.
$$
If the reverse Lions' condition holds (that is; \eqref{matrix} holds with opposite inequality), then 
$$
\frac{\partial}{\partial j}\left[v(x,j)\right] \geq 0.
$$

\end{proposition}

\begin{proof}
Differentiating \eqref{optv}, with respect to $j$, we get
\begin{equation*}
\frac{\partial}{\partial j}\left[v(x,j)\right]=
\frac{1}{m} - \frac{j}{m^2} \frac{\partial m}{\partial j}.
\end{equation*}
From \eqref{djsol} and using that $v(x,j)=-D_pH(x,u_x(x,j),m(x,j))$, we have
\begin{equation}\label{dmdj}
\frac{\partial m(x,j)}{\partial j}= \frac{-D_pH}{M} =\frac{j}{m M},
\end{equation}
Hence, 
$$
\frac{\partial}{\partial j}\left[v(x,j)\right]=
\frac{1}{m}\left[ 1-\frac{j^2}{m^2}\frac{1}{M} \right]=
- \frac{vD^2_{pm}H+D^2_{pp}HD_mH}{M}.
$$
In the separable case, we have
\begin{equation}\label{dvdj}
\frac{\partial}{\partial j}\left[v(x,j)\right]=-\frac{D^2_{pp}Hg^{\prime}(m)}{M}.
\end{equation}
If Lions' condition holds, then by Proposition \ref{prop:mon}, $M \geq 0$. Thus, 
\begin{equation*}
\frac{\partial}{\partial j}\left[v(x,j)\right] \leq 0.
\end{equation*}
If the reverse Lions' condition holds, we get the opposite inequality. 
\end{proof}

Now, we study the relation between the velocity and the density through the quantity 
$\frac{\partial m}{\partial v}$,  
which  corresponds to the uncongested model when it is greater or equal to zero.

\begin{proposition}
For $j>0$, in the separable case, we have  
\begin{equation}\label{dmdv}
\frac{\partial m}{\partial v}= -\frac{v}{D^2_{pp}Hg^{\prime}(m)}.
\end{equation}
\end{proposition}

\begin{proof}
We have
$$
\frac{\partial m}{\partial v}=
\frac{\partial m}{\partial j}. \frac{\partial j}{\partial v}.
$$
By \eqref{dmdj} and \eqref{dvdj}, we get 
$$
\frac{\partial m}{\partial v}=
-\frac{j}{mM} 
.\frac{M}{D^2_{pp}Hg^{\prime}(m)},
$$
which gives \eqref{dmdv} after simplification.
\end{proof}

\begin{remark}
From \eqref{dmdv}, we see that if $g(m)$ increasing, $m$ decreases with $v$.
Whereas if $g(m)$ decreasing, $m$ increases with $v$.
\end{remark}

\paragraph{A summary}

Now, we summarize the results in this section. 
In highways or pedestrian traffic, we expect that as $j$ increases, $m$ increases and $v$ decreases. 
For separable $H$, and $j>0$, if Lions' condition \eqref{matrix} holds, then
\begin{enumerate}

\item $c_{01}$ increases as $j$ increases, see Proposition \ref{prop:mon}. 
\item $m$ increases as $j$ increases, 
see Proposition \ref{prop:mon}. 
\item $v$ decreases as $j$ increases. 
see Proposition \ref{Prop:v-j}. 
\end{enumerate}

However, Lions' condition is not needed  for the previous points to hold. Example \ref{ex:mono} gives conditions under which $1$ and $3$ hold without monotonicity. 
However, $2$ does not hold; see Remark \ref{remark:mono}.


\section{Calibration of MFGs}
\label{sec:MFG_calibration}    

In this section, we address Problem \ref{prob:p001}; that is, the precise correspondence between MFG and Wardrop models. 
To study it in more detail, we consider the following two inverse problems: 
\begin{problem}\label{prob:p1}
Let $W$ be a Wardrop model on a single edge,  identified with the interval $[0,1]$ with  $c : \mathbb{R}_0^+ \to \mathbb{R}^+$ a reversible current-dependent travel cost.
  Find $L$ such that, for any current $j$,
  \begin{equation}\label{L-c}
  \int_0^{T(j)} L(x, v, m) ds = c(j), 
    \end{equation}
     where $T(j)>0$ is the crossing time for an agent moving at optimal speed.
\end{problem}

\begin{problem}\label{prob:p2}
Let $W$ be a Wardrop model on a single edge,  identified with the interval $[0,1]$ with  $c : \mathbb{R}_0^+ \to \mathbb{R}^+$ a reversible current-dependent travel cost.
  Find $L$ such that, for any current $j$,
  \begin{equation}\label{L-c}
  \int_0^{T(j)} L(x, v, m) ds = c(j), \  \text{and} \ T(j)=c(j). 
    \end{equation}
     where $T(j)>0$ is the crossing time for an agent moving at optimal speed.
\end{problem}

While Problem \ref{prob:p1} may seem more natural, Problem \ref{prob:p2} is more straightforward to calibrate as it is enough to measure average speed $v$ as a function of current to get the cost $c(j)$. 
In general, we are not aware of any other way to assign a travel cost depending on the current.  

The solution to these problems is not unique, and any solution can be used to generate a MFG model associated with the cost $c$. 
One of the reasons for non-uniqueness is that there are no microscopic effects in Wardrop models. 
In Section \ref{sec:min-travel-cost as mfg},  we
work with a specific class of Lagrangians and illustrate how to solve these problems. 
In particular, we work with $x$-independent Lagrangians. 
There is no uniqueness in this class; see Example  \ref{exL}.



\subsection{Identification of Wardrop cost problems with mean-field games}
\label{sec:min-travel-cost as mfg}

In Problem \ref{prob:p1}, we identify a corresponding MFG model whose travel cost on an edge, for a given current $j$, is $c(j)$. 
We assume that agents travel from 0 to 1. Thus, the current is positive.
We do not expect to recover the MFG microstructure from the Wardrop model's macrostructure, so we  consider Lagrangians without $x$-dependency. 
To have reversible costs, we consider
 Lagrangians which are even in the velocity.  
 More concretely, we restrict our choice to 
\begin{equation}\label{lagrangian}
L(x,v, m) = m^\alpha \mathcal{L}(v) + g(m),
 \end{equation} 
with $\mathcal{L}$ convex and even.
 This class is broad enough to include interesting examples and specific enough that we can get closed formulas/ models for the cost. 
Because the Lagrangian is even and the cost is reversible, solving the problem for $j>0$ is enough. 
Moreover, since there are no location preferences in the edge, the incremental cost of the travel is the same everywhere; that is, $u_x$ is constant in $x$. 
Therefore, the density is constant,  i.e.,  $m=m(j)$, as we can see from the MFG system.  
The corresponding MFG is equivalent to 
\begin{equation}
  \label{eq:HJ on interval}
\begin{cases}
  m^\alpha \mathcal{H}\left(\frac{u_x}{m^\alpha}\right) = g(m)\\
 -m\mathcal{H}'\left(\frac{u_x}{m^\alpha}\right) = j,
\end{cases}
\end{equation}
where $\mathcal{H}$ is the Legendre transform of $\mathcal{L}$. 
Following the same steps as before, we find 
\begin{equation}\label{u_x}
  u_x = -m^\alpha \mathcal{L}' \left(\frac{j}{m}\right). 
\end{equation}
From the second equation in \eqref{eq:HJ on interval} and the relation $- \mathcal{L}'(- \mathcal{H}'(p)) = p$,
we transform the MFG system into the algebraic system
\begin{equation}\label{eq:algebraic}
\begin{cases}
m^\alpha \mathcal{H}\left(- \mathcal{L}' \left(\frac{j}{m}\right)\right) = g(m)\\
c(j) = u(0) - u(1) = m^{\alpha}\mathcal{L}'\left(\frac{j}{m}\right), 
\end{cases}
\end{equation}
provided $m>0$.
Suppose that the last equation in \eqref{eq:algebraic} is solvable for $m$ and $j$, i.e. there exists an invertible function $\Psi_c$ such that
\begin{equation}\label{psic}
m = \Psi_c(j).
\end{equation}
Now, let $\Phi_j(m):=m^{\alpha}\mathcal{L}'\left(\frac{j}{m}\right)$. 
Then, by \eqref{psic} and the last equation in \eqref{eq:algebraic}, we have 
$$
\Psi_c(j)=m = \Phi_j^{-1}(c(j)). 
$$
Hence, we get 
$$
\Psi_c^{-1}(j)= c^{-1}(\Phi_j(j)). 
$$
Hence, $\Psi_c$ is well defined, at least locally, if $m \rightarrow \Phi_j(m)$ is strictly monotone in $m$ for every $j$. If, in addition, $c$ is strictly monotone, then $\Psi_c$ is invertible. 
Now, inverting \eqref{psic} and substituting in the first equation of \eqref{eq:algebraic}, we get 
\begin{equation}\label{mmm}
m^\alpha \mathcal{H}\left(- \mathcal{L}' \left(\frac{\Psi_c^{-1}(m)}{m}\right)\right) = g(m). 
\end{equation}
This relation identifies $g(m)$ in terms of $c,\mathcal{L}'$ and $\mathcal{H}$. 





Next, we show an example that applies the previous discussion to solve the inverse problem.  


\begin{example}\label{exL}
Let $\alpha\neq 1$, and $\beta>0$.
  Consider the Lagrangian, $\mathcal{L}(v) = \beta\frac{v^2}{2}$, with corresponding Hamiltonian, $\mathcal{H}(p)=\frac{p^2}{2\beta}$.
  The system \eqref{eq:algebraic} becomes
\begin{equation}
\begin{cases}
    m^{\alpha-2} \beta j^2 = 2 g(m)\\
    c(j) = m^{\alpha -1} \beta j, 
    \end{cases}
    \end{equation}
    provided $m>0$.
  The second equation gives, for $\alpha \neq 1$, $m$ as a function of $j$:
  \begin{equation}\label{Psi}
    m = \left(\frac{c(j)}{\beta j}\right)^{\frac{1}{\alpha-1}} =: \Psi_c(j).
   \end{equation}
  Thus, if $\Psi_c$ is invertible,
  we obtain 
  $$
    g(m) =\frac{\beta}{2}  m^{\alpha -2} \left(\Psi_c^{-1}(m)\right)^2.
  $$
Now, consider a linear cost, $c(j) = c_1+c_2j$. By \eqref{Psi}, the density is 
$$
m= \left(  \frac{c_1}{\beta j} + \frac{c_2}{\beta}  \right) ^{\frac{1}{\alpha-1}}. 
$$
Assuming $c_1, c_2 \geq 0$, the range of $m$ is 
\begin{equation}\label{range}
\begin{cases}
\left[ \left( \frac{c_2}{\beta} \right)^{\frac{1}{\alpha-1}} ,+ \infty  \right), \ \ \text{for} \ \ \alpha >1. \\ 
\left( 0, \left( \frac{c_2}{\beta} \right)^{\frac{1}{\alpha-1}} \right), \ \ \text{for} \ \  \alpha <1. 
\end{cases}
\end{equation}
For our linear cost, $\Psi_c$ is invertible and we have 
$$
g(m)=\frac{\beta m^{\alpha-2}c_1^2}{2 \left( \beta m^{\alpha-1}-c_2 \right)^2}, 
$$
which can be written as follows
$$
g(m) = \frac{1}{2\beta}  \frac{c_1^2}{\left(m^{\alpha/2} - \frac{c_2}{\beta} m^{1-\alpha/2}\right)^2}.
$$
\begin{IMG}
{"singularG1",
Plot[(2*m^(2 - \[Alpha])*(m^(-1 + \[Alpha])*\[Beta] - Subscript[c, 2])^2)/(\[Beta]*Subscript[c, 1]^2)/. {\[Alpha] -> .1, \[Beta] -> 1, Subscript[c, 1] -> 1, 
Subscript[c, 2] -> 0.9}, {m, 0, 2},AxesLabel -> {m, 1/g[m]}
]}(*\]\]\]\]\]\]*)
\end{IMG}
\begin{figure}[ht]
  \centering
  \includegraphics[scale=.6]{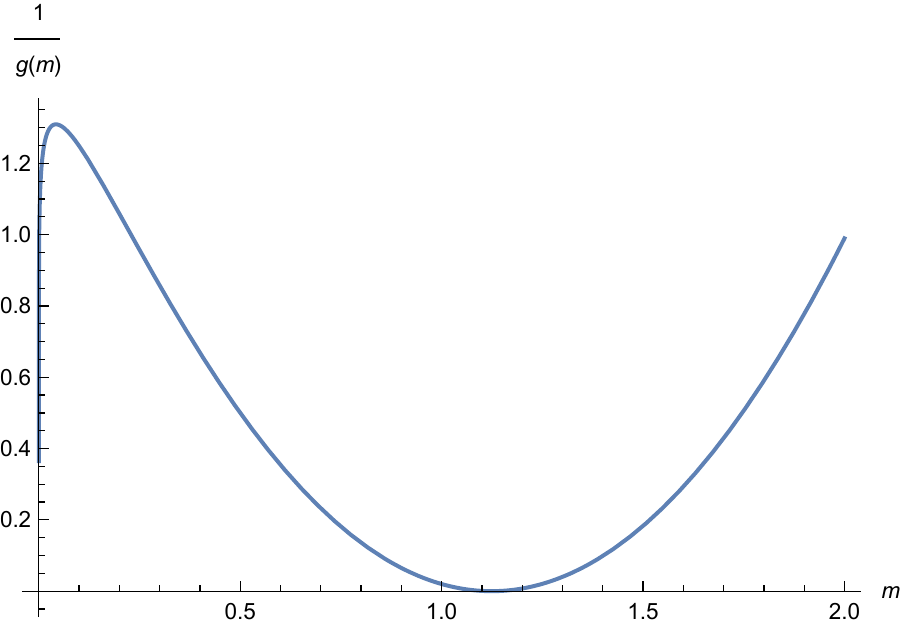}
  \caption{For $\alpha=0.1$, $\frac{1}{g(m)}$ touches 0 at two points.}
  \label{fig: 1/g}
\end{figure}
\begin{IMG}
{"singularG2",
Plot[(2*m^(2 - \[Alpha])*(m^(-1 + \[Alpha])*\[Beta] - Subscript[c, 2])^2)/(\[Beta]*Subscript[c, 1]^2)/. {\[Alpha] -> 3, \[Beta] -> 1, Subscript[c, 1] -> 1, 
Subscript[c, 2] -> 0.9}, {m, 0, 2},AxesLabel -> {m, 1/g[m]}
]}(*\]\]\]\]\]\]*)
\end{IMG}
\begin{figure}[ht]
  \centering
  \includegraphics[scale=.6]{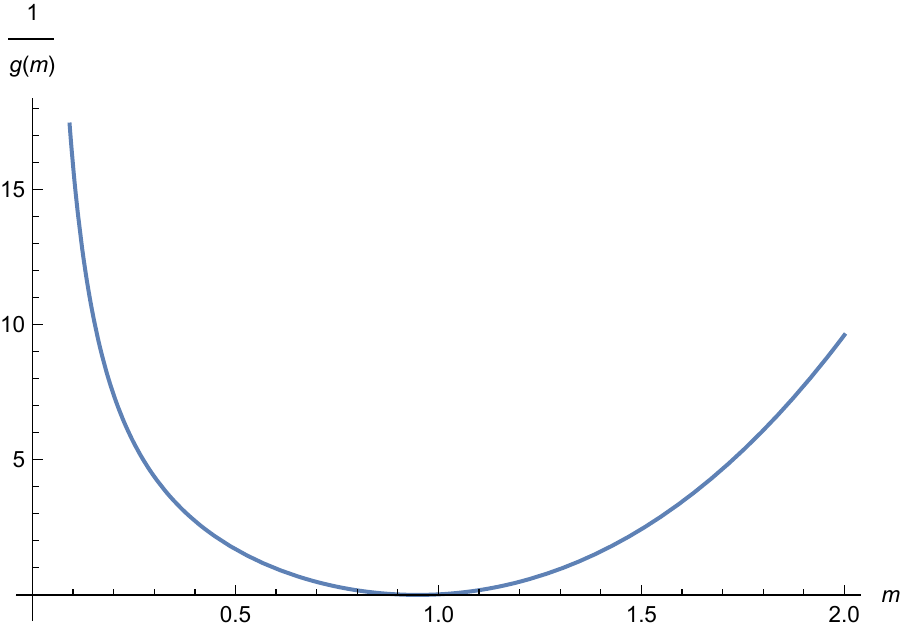}
  \caption{For $\alpha=3$, $\frac{1}{g(m)}$ touches 0 at one point.}
  \label{fig2: 1/g}
\end{figure}

Note that such $g$ is non-monotone.
For $\alpha<1$, it has one singularity at 0 and another one at $m_0=\left( \frac{c_2}{\beta} \right)^{\frac{1}{\alpha-1}}$. 
This last value $m_0$ is the maximal density. 
 For  $\alpha>1$,  there is a singularity at  $m_0$, but it is a minimal density in this case.  
This is illustrated in Figure \ref{fig: 1/g} where we plot $\frac{1}{g(m)}$ for 
$\alpha = 0.1$ and Figure \ref{fig2: 1/g} where we plot $\frac{1}{g(m)}$ for 
$\alpha = 3$, with  $\beta = 1$, $c_1=1$ and $c_2=0.9$ in both figures. 
\end{example}

The preceding example illustrates an important point: even with 
the simple linear current-dependent cost, we obtain a coupling function $g(m)$ that is non-monotone which falls outside standard
MFG theory. 
If $\alpha <1$, $g(m)$ is neither  decreasing nor increasing for the values of $m$ in \eqref{range}. 
For $\alpha >1$, $g(m)$ is decreasing for the values of $m$ that satisfy \eqref{range}. 


\subsection{Identification of Wardrop time problems with mean-field games}
\label{sec: braess}


Now, we consider Problem \ref{prob:p2}. 
As discussed in Section \ref{analysisofv}, the optimal velocity is $v=\frac{j}{m}$.
Because the edge has length $1$ and the velocity, $v$, is constant, we have
\begin{equation*}
 T=\frac{m}{j}. 
 \end{equation*}
 To address this problem, we consider $L$ as in \eqref{lagrangian}. 
 In this case, we show how to determine both $\mathcal{L}$ and $g$. 
 Arguing as in Section \ref{sec:min-travel-cost as mfg}, we obtain 
\begin{equation}\label{eq: travel time system}
  \begin{cases}
    m^\alpha \mathcal{H}\left(-\mathcal{L}'\left(\frac{j}{m}\right)\right)=g(m)& \\
    m^\alpha \mathcal{L}'\left(\frac{j}{m}\right) =c(j)& \\
    c(j) = \frac{m}{j}, 
  \end{cases}
\end{equation}
provided $m>0$.
As before, we study the case $j>0$. 
From the last equation in \eqref{eq: travel time system}, we get 
\begin{equation}\label{mjc}
m=jc(j), \ \forall  j>0. 
\end{equation}
Let $\Psi_c(j)=jc(j)$, if $c$ is positive and increasing, then $\Psi_c(j)$ is invertible. 
Hence, 
$$
j=\Psi_c^{-1}(m). 
$$
Moreover, in this case, $m$ can take any value in $\mathbb{R}^+$.
Using the prior identity, we substitute $j$ in the second equation of \eqref{eq: travel time system}. Next, we use the result in the first equation of \eqref{eq: travel time system} to get 
\begin{equation}\label{g(m)}
g(m) = m^\alpha \mathcal{H}\left(-\frac{c(  \Psi_c^{-1}(m))}{m^\alpha}\right).
\end{equation}
Up to here, $\mathcal{L}$ is an unknown. 
Now, we substitute $m$ from \eqref{mjc} in  the second equation in \eqref{eq: travel time system} to get 
$$
  \mathcal{L}'\left(\frac{1}{c(j)}\right) = \frac{c(j)}{\left(jc(j)\right)^\alpha}.
$$
Finally, recalling that $v = 1/c(j)$, we have
\begin{equation}\label{Lprime}
  \mathcal{L}'\left(v\right) = \frac{v^{\alpha-1}}{\left(c^{-1}\left(\frac{1}{v}\right)\right)^\alpha}.
\end{equation}
If $c$ is increasing, the preceding expression defines $\mathcal{L}'(v)$ for   
$\lim_{j\to \infty}\frac{1}{c(j)}\leq  v \leq\frac{1}{c(0)}$.  
Thus, $\frac{1}{c(0)}$ is the maximal velocity. 
So, to solve the inverse problem, we first use \eqref{Lprime} to identify $\mathcal{L}$, compute its Legendre transform $\mathcal{H}$ and use \eqref{g(m)} to obtain $g$. 

In the following example, we look at a simple cost structure, affine in $j$, and show that non-monotone MFGs may arise. 

\begin{example}
  Suppose $c(j) = 1+j$.
  We want to find $\mathcal{L}$ and $g$ for which the cost is $c(j)$.
  The system \eqref{eq: travel time system} simplifies to
  \begin{equation*}
    \begin{cases}
      g(m) = m^\alpha \mathcal{H} \left(-\frac{2m^{1-\alpha}}{ \sqrt{1 + 4 m} -1}\right),\\
      \mathcal{L}'(v) = \frac{v^{2\alpha-1}}{(1-v)^\alpha}\\
      v=\frac{1}{1+j}\\
      m= j (1+j).
    \end{cases}
  \end{equation*}
The Lagrangian is convex if $\alpha\geq \frac 1 2$. 
However, the resulting
coupling may fail to be monotone, as Figure \ref{fig: lagrangian and derivatives} illustrates. 
\begin{IMG}
{"plotglastsection",
	a = 0.6;
	c[j_] := 1 + j;
	Solve[m^a LLp[1/(1 + j)] == 1 + j /. m -> (j/v) /. 
	Solve[v == 1/c[j], j] // FullSimplify, LLp[v]] // PowerExpand;
	Lp[v_] := (1 - v)^-a v^(-1 + 2 a);
	LL[v_?NumberQ] := NIntegrate[Lp[z], {z, 0, v}];
	(*Plot[Lp[z],{z,0.001,0.999},AxesLabel\[Rule]"L'",PlotLabel\[Rule]{\
		"alpha", a}]*)
	
	VV[p_?NumberQ] := v /. First[FindRoot[p + Lp[v], {v, 0.015}]];
	HH[p_] := -p VV[p] - LL[VV[p]];
	(*galt[m_?NumberQ]:=(m^a HH[-2m^(1-a)/(Sqrt[1+4m]-1)])*)
	
	g[m_?NumberQ] := (m^a HH[-(1/2) m^-a (1 + Sqrt[1 + 4 m])]);
	Plot[g[m], {m, 0.01, 3}, AxesLabel -> {m, g[m]}, 
	PlotRange -> All]
	(*Plot[galt[m],{m,0.01,3},AxesLabel\[Rule]"g[m]",PlotLabel\[Rule]{\
		"alpha", a},PlotRange\[Rule]All]*)}(*\]\]\]\]\]*)
\end{IMG}	  
  \begin{figure}[ht]
   \centering
    \includegraphics[scale=.7]{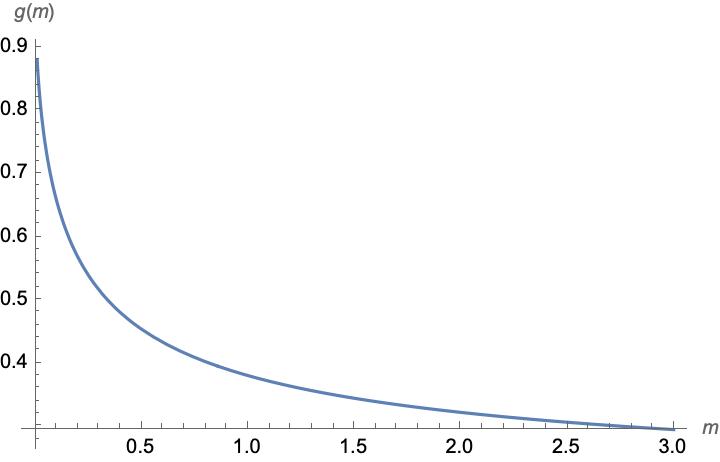}
    \caption{The coupling $g(m)$ decreasing with $m$, for $\alpha=0.6$. }
    \label{fig: lagrangian and derivatives}
  \end{figure}
\end{example}

\subsection{Example: Braess paradox}

Now, we present an application for the calibration of MFGs in Wardrop problems. 
In particular, we show how the celebrated Braess paradox can arise in MFGs.

\begin{example}[Braess paradox]
\label{ex:Braess}
  \begin{IMG}
    {"Example", 
      Graph[{Labeled[1 \[DirectedEdge] 2, Subscript["e", 1] "c(j)=45,"], 
      Labeled[1 \[DirectedEdge] 3, Subscript["e", 2] "c(j)=|j|/100,"], 
      Labeled[2 \[DirectedEdge] 4, Subscript["e", 3] "c(j)=|j|/100,"], 
      Labeled[3 \[DirectedEdge] 4, Subscript["e", 4] "c(j)=45,"], 
      Labeled[2 \[DirectedEdge] 3, Subscript["e", 5] "c(j)=0,"],
      Labeled[5 \[DirectedEdge] 1, Subscript["e", OverTilde[1]]], 
      Labeled[4 \[DirectedEdge] 6, Subscript["e", OverTilde[4]]]}, 
      VertexLabels -> {1 -> Subscript["v", 1], 2 -> Subscript["v", 2], 
      3 -> Subscript["v", 3], 4 -> Subscript["v", 4], 
      5 -> Subscript["v", OverTilde[1]], 
      6 -> Subscript["v", OverTilde[4]]},
      EdgeStyle -> {5 \[DirectedEdge] 1 -> Dashed, 
      4 \[DirectedEdge] 6 -> Dashed}, GraphLayout -> "SpringEmbedding"](*\]\]\]\]\]\]\]\]\]*)
    }
    \end{IMG}
    \begin{figure}[h!]
      \centering
      \includegraphics[scale=1]{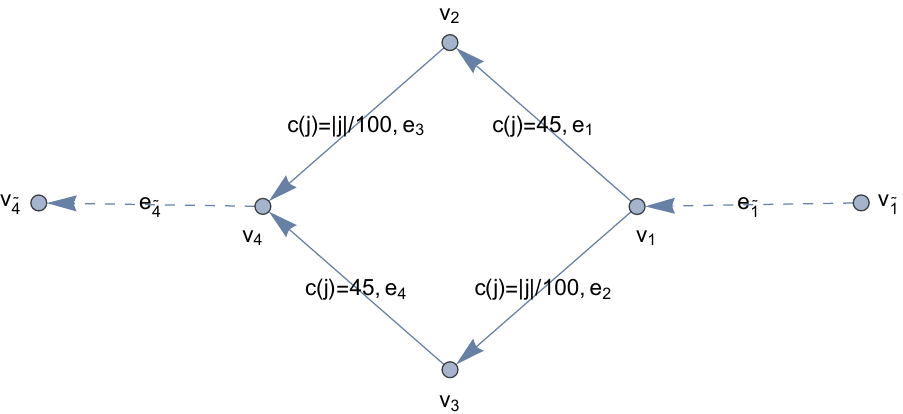}
     \caption{Braess paradox network}
     \label{braessgraph}
     \end{figure}

  Consider the network in Figure \ref{braessgraph} with the given costs. 
  Agents want to move from the entrance vertex $v_1$ to the exit vertex $v_4$ through the network with the least travel cost. 
  Assume that the exit cost is zero and the entry current is $4000$.  
At  $v_1$, 
half of the agents will go through the edge $e_1$, and the other half will use the edge $e_2$ at the cost of $65$.  
If we add a new edge $e_5=(v_2,v_3)$, as Figure \ref{braessgraph2} illustrates, with zero travel cost, some agents will choose to take the path $e_2, e_5, e_3$.
If all agents are taking this path, this is a Wardrop equilibrium. 
By adding $e_5$, the cost increases for everyone. This is the well-known Braess paradox. 
       \begin{figure}[h!]
      \centering
      \includegraphics[scale=1]{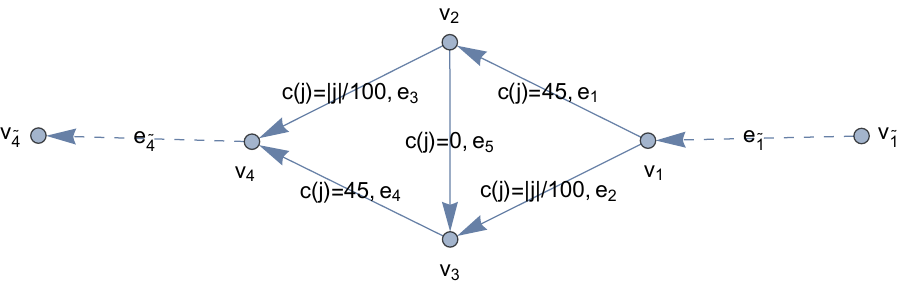}
     \caption{Braess paradox network with edge in the middle.}
     \label{braessgraph2}
     \end{figure}

\begin{ART}[h!]
{c[j_] = j/100,
	Solve[m^a LLp[1/(c[j])] == c[j] /. m -> (j/v) /. 
	Solve[v == 1/c[j], j], LLp[v]] // PowerExpand,
	Integrate[100^-a v^(-1 + 2 a), v],
	q = First[qq /. Solve[1/qq + 1/(2 a) == 1, qq]],
	H[p_] = 100^a p^q/q,
	g[m_] = m^a H[-c[First[j /. Solve[c[j] == m/j, j]]]/m^a] // 
	PowerExpand // FullSimplify
}	
\end{ART}

Since we proved in Section \ref{WtoMFg} that  MFGs can be recovered through Wardrop problems, 
a natural question is which MFGs give rise to the particular cost structure of the Braess paradox? 
The answer to this question is given by the following computations that determine the MFG for the Braess paradox. 
For the edges $e_3$ and $e_2$, 
we obtain, for $\alpha>\frac{1}{2}$, 
\[
L(v)=\frac{|v|^{2\alpha}}{ 2\alpha 100^\alpha } \qquad H(p)=100^\alpha \frac{|p|^{\kappa}}{\kappa},
\]
with $\kappa=\frac{2\alpha}{2 \alpha-1}$, and
\[
g(m)=C_\alpha, \qquad C_\alpha=
\frac{100^{(\alpha-1)\kappa}}{\kappa}.
\] 
The constant costs do not fit the previous calibration. 
However, in $e_1$ and $e_4$, for small $\epsilon$, we can set the cost to  
$$
c(j)=45+\epsilon |j|.
$$
Similarly, in the edge $e_5$, we set the cost to 
$$
c(j)=\epsilon |j|. 
$$ 


  \end{example}


%
%
%
%
%
%

\smallskip

{\bf Acknowledgments}: 
The authors were partially supported by King Abdullah University of Science and Technology (KAUST) baseline funds and KAUST OSR-CRG2021-4674.


F. Al Saleh acknowledges King Faisal University for their support.

\bibliographystyle{alpha}

\bibliography{mfg.bib}

\end{document}